\documentclass[11pt,a4paper]{article}
\usepackage{amsmath, amsfonts, amssymb,amsthm}
\usepackage{a4wide}
\usepackage{parskip}
\usepackage{enumitem}
\usepackage{xcolor}
\usepackage{pstricks}
\usepackage{hyperref}
\usepackage{latexsym}
\usepackage{cite}
\def\QED{\hfill {$\square$}\goodbreak \medskip}

\usepackage[left=1in, right=1in,top=1in,bottom=1in]{geometry}
\allowdisplaybreaks

\newcommand{\be} {\begin{equation}}
	\newcommand{\ee} {\end{equation}}
\newcommand{\bea} {\begin{eqnarray}}
	\newcommand{\eea} {\end{eqnarray}}
\newcommand{\Bea} {\begin{eqnarray*}}
	\newcommand{\Eea} {\end{eqnarray*}}

\newcommand{\ba} {\beta}
\newcommand{\de} {\delta}
\newcommand{\ga} {\gamma}
\newcommand{\Ga} {\Gamma}

\newcommand{\La} {\Lambda}

\newcommand{\e} {\epsilon}
\def\R{{\mathbb R}}
\def\N{{\mathcal N}}

\def\R{{\mathbb R}}
\def\Z{{\mathbb Z}}
\def\N{{\mathbb N}}
\def\F{{\mathcal F}}
\def\H{{\mathbb {H}^N}}

\def\nah{\nabla_{\mathbb{H}}}
\def\I{I_{\lambda}}
\def\N{{N_\lambda}}

\numberwithin{equation}{section}

\newtheorem{theorem}{Theorem}[section]
\newtheorem{rem}{Remark}[section]
\newtheorem{lemma}{Lemma}[section]
\newtheorem{co}{Corollary}[section]

\makeatother
\newtheorem{prop}{Proposition}[section]

\numberwithin{equation}{section}
\def\proof{\noindent{\textbf{Proof. }}}

\begin{document}
	\setlength{\abovedisplayskip}{3pt}
	\setlength{\belowdisplayskip}{3pt}
	\date{}
	\title{ Critical Quasilinear Schr\"odinger Equations on the Heisenberg Group: Existence and Nonexistence }
    \author{ {\bf Ankit Mishra$\,$\footnote{e-mail: {\tt ankitmishra.rs.mat23@iitbhu.ac.in, ankitmishra18295@gmail.com}}, Divya Goel$\,$\footnote{e-mail: {\tt divya.mat@iitbhu.ac.in}}} \\ Department of Mathematical Sciences, Indian Institute of Technology (BHU),\\ Varanasi 221005, India.}
	
\maketitle
\begin{abstract}
		We study the quasilinear Schr\"odinger equation
		\begin{align*}
			-\Delta_{\mathbb{H}} u +V(\xi)u-\Delta_{\mathbb{H}} (\left|u\right|^{2\alpha})\left|u\right|^{2\alpha-2} u= \lambda \left|u\right|^{q-2}u + \left|u\right|^{p-2}u \quad \text{ in } \mathbb{H}^N,
		\end{align*}
		where $\Delta_{\mathbb{H}}$ is the Kohn Laplacian on the Heisenberg group $\mathbb{H}^N$, $4\alpha<q<p \leq 2\alpha Q^*$,  $\alpha>\frac12$, and $2\alpha Q^{*}$ is the critical exponent, $Q=2N+2$ being the homogeneous dimension and $Q^{*}=\frac{2Q}{Q-2}$. For $p=2\alpha Q^{*}$ and $\lambda>0$, we obtain a nontrivial solution, assuming that the potential is bounded below by a positive constant and is either asymptotically constant from above or invariant under a discrete subgroup of $\mathbb{H}^N$. In the opposite direction, we prove a Poho\v{z}aev identity for $\Delta_{\mathbb{H}}$ and combine it with the Nehari identity, obtaining a family of identities from which the quasilinear energy disappears exactly at the exponent $2\alpha Q^{*}$. This yields a nonexistence theorem under a monotonicity condition on the potential with respect to anisotropic dilations and shows that no nontrivial solution exists for $\lambda \leq 0$; the sign of the subcritical perturbation determines solvability. Along the way, we show that every weak solution is bounded and decays exponentially in the Kor\'anyi gauge.

		\medskip

		\noindent \textbf{Keywords:}

		\noindent \textit{Quasilinear Schr\"odinger equation, Heisenberg group, Kohn--Laplacian, Poho\v{z}aev identity, Non-existence, Variational method.}

		\medskip

		\noindent \textbf{MSC 2020:} 35R03, 35H20, 53C17, 35Q55.
	\end{abstract}

	\section{Introduction}

	Quasilinear Schr\"odinger equations couple the Laplacian to nonlinear functions of the solution and of its gradient. These types of equations arise in the description of superfluid films \cite{kurihara1981large}, in plasma physics, and in nonlinear optics, where the wave function is subject to both a nonlinear potential and nonlinear dispersion. A stationary equation of this type reads
	\begin{equation}\label{1.1}
		-\Delta u + V(x)u - \kappa\,\Delta\!\left(\left|u\right|^{2\alpha}\right)\left|u\right|^{2\alpha-2}u = h(u)\quad\text{in } \mathbb{R}^{m}, \ m \ge 3,
	\end{equation}
	and its solutions produce standing waves $\psi(t,x)=e^{-iEt}u(x)$, $E \in \R$, of the corresponding evolution equation.

	When $\kappa=0$, the problem is semilinear and has been studied under a wide range of assumptions on $V$ and $h$; see \cite{medeiros2008nonhomogeneous, giacomoni2005multiplicity, lam2012existence}. For $\kappa \neq 0$, the term $\Delta(\left|u\right|^{2\alpha})\left|u\right|^{2\alpha-2}u$ is neither convex nor of lower order, and outside dimension one there is no function space on which the associated energy is both finite and of class $C^1$ \cite{poppenberg2002existence}. Direct variational arguments are therefore unavailable, and the literature proceeds by perturbation, by constrained minimisation, or by a nonlinear change of variable.

The first variational existence result for  $\alpha=1$ is by Poppenberg, Schmitt and Wang \cite{poppenberg2002existence}, who minimized under a constraint in dimension one and for radial potentials in higher dimensions; their solution carries an undetermined Lagrange multiplier in front of the nonlinearity. Colin and Jeanjean \cite{colin2004solutions} removed the multiplier by a change of variable that turns \eqref{1.1} into a semilinear equation on $H^1(\R^{m})$, to which the results of Berestycki and Lions \cite{etde_5937557} apply for $m=1$ and $m \ge 3$, and those of Berestycki, Gallou\" et and Kavian \cite{berestycki1983equations} for $m=2$. The subcritical regime is treated in \cite{floer1986nonspreading, rabinowitz1992class, strauss1977existence}, and
    and do \'O, Miyagaki, and Soares \cite{miyagaki2010soliton} obtained positive classical solutions for the critical nonlinearity
$h(u)=\lambda\left|u\right|^{q-2}u+\left|u\right|^{2\cdot 2^{*}-2}u$ with $\lambda>0$.
    
For the quasilinear problem the critical exponent is $2\cdot 2^*$ rather than $2^*$. Liu, Wang, and Wang \cite{liu2004solutions} identified $2\cdot 2^{*}$, rather than $2^{*}$, as the threshold for \eqref{1.1} for$\alpha=1$.  Using the variational identity of Pucci and Serrin \cite{pucci1986general}, authors proved that no positive solution exists in $H^1(\R^{m})$ with $u^2\left|\nabla u\right|^2 \in L^1(\R^{m})$ when $p \ge 2\cdot2^{*}$. This nonexistence result shows that $2\cdot 2^*$ is a threshold for \eqref{1.1} and not a restriction of a particular method. Theorem \ref{thm:nonex} establishes
the corresponding result on $\mathbb{H}^N$.

	For general $\alpha>\frac12$, the first variational result is due to Liu and Wang \cite{liu2003soliton}, who studied
	\begin{equation}\label{1.2}
		-\Delta u + V(x)u - \Delta\!\left(\left|u\right|^{2\alpha}\right)\left|u\right|^{2\alpha-2}u = \lambda \left|u\right|^{p-2}u \quad \text{in } \mathbb{R}^{m},
	\end{equation}
	for $2<p<2\alpha 2^{*}$, again by a Lagrange multiplier argument with $\lambda$ undetermined. Adachi and Watanabe \cite{adachi2012uniqueness} settled the uniqueness of positive solutions of \eqref{1.2} for constant $V$, by reducing to a semilinear problem and invoking classical uniqueness results, while Alves, Carri\~ao, and Miyagaki \cite{alves2007soliton} treated a one-dimensional problem for a more general operator and produced two solutions of a perturbed equation. Chen obtained multiplicity results \cite{chen2015multiple}, and Wu \cite{wu2014multiple} produced a positive solution, a negative solution, and a sequence of high-energy solutions for $\frac12<\alpha\le1$ via a dual variational method. A unified treatment of the class $-\Delta u - u\,l^{\prime}(u^2)\Delta l(u^2)=h(u)$, which contains \eqref{1.2} as the case $l(s)=s^{\alpha}$, was given by Shen and Wang \cite{shen2013soliton}.

	In dimension two, the  nonlinearity for the
semilinear equation is $\exp(\beta|u|^2)$,  whereas in the presence of the term $\Delta(u^2)u$ the correct threshold is $\exp(\beta \left|u\right|^{4})$ \cite{miyagaki2007soliton}, and Trudinger--Moser inequalities adapted to the transformed functional are required; see \cite{miyagaki2007soliton, moameni2007class, do2009semi} and the references therein.

    All the work described so far is set in $\R^{m}$. We are concerned with the same questions on the Heisenberg group $\H$, the simplest non-commutative, sub-Riemannian manifold and the model case for analysis on stratified Lie groups. Two features make it a natural setting for critical growth problems. First, the Kohn Laplacian $\Delta_{\mathbb{H}}$ admits a Sobolev-type inequality, due to Folland and Stein \cite{folland1974estimates}, whose critical exponent is $Q^{*}=2Q/(Q-2)$, where $Q=2N+2$ is the homogeneous dimension of $\H$. Second, the extremals of that inequality are known explicitly. The loss of compactness at the exponent $Q^{*}$ is therefore of the same nature as in the Brezis--Nirenberg problem \cite{brezis1983positive}, and semilinear equations with this growth on $\H$ have been studied by several authors. Garofalo and Lanconelli \cite{garofalo1992existence} proved the first existence and nonexistence results, and Citti \cite{citti1995semilinear} treated the critical Dirichlet problem on bounded domains.  Later, Lanconelli and Uguzzoni \cite{lanconelli2000nonexistence} and Loiudice \cite{loiudice2007semilinear} obtained Nonexistence in unbounded domains.

	We study
	\begin{equation}\tag{$P_{\alpha}$}\label{eq:P}
		-\Delta_{\mathbb{H}} u + V(\xi)u
		- \Delta_{\mathbb{H}}\!\left(\left|u\right|^{2\alpha}\right)\left|u\right|^{2\alpha-2}u
		= \lambda \left|u\right|^{q-2}u + \left|u\right|^{p-2}u,
		\qquad \xi \in \H,
	\end{equation}
	with $\lambda>0$, $\alpha>\frac12$, $V:\H\to\R$ positive, and
	\begin{align*}
		4\alpha  < q < p \le 2\alpha Q^{*}.
	\end{align*}
	The exponent $2\alpha Q^{*}$ is critical for \eqref{eq:P} in the same sense that $2\cdot2^{*}$ is critical for \eqref{1.1}. The potential is assumed to satisfy
	\begin{description}
		\item[$(V_1)$] $V$ is continuous and there is $V_0>0$ with $V(\xi) \geq V_0$ for all $\xi \in \H$,
	\end{description}
	together with one of
	\begin{description}
		\item[$(V_2)$] there is $V_\infty$ such that $\lim_{\left|\xi\right|_{\mathbb{H}} \to \infty} V(\xi)= V_\infty$, with $V(\xi) \leq V_\infty$ for all $\xi$ and $V \not\equiv V_\infty$;
		\item[$(V_2^\prime)$] $V(\gamma \circ \xi)=V(\xi)$ for every $\xi \in \H$ and every $\gamma$ in the discrete subgroup $\Ga_{\mathbb{Z}}=\Z^{N}\times\Z^{N}\times\Z$ of $\H$.
	\end{description}

	Condition $(V_1)$ gives only a lower bound. Under $(V_2)$ we have $V \leq V_\infty$, and under $(V_2^\prime)$ the potential is continuous and periodic, hence bounded. In both cases $V$ is bounded, and we write
	\begin{align*}
		V_M := \sup_{\xi \in \H} V(\xi) < \infty.
	\end{align*}
	The existence theorem also requires the exponent $q$ to lie above the threshold
	\begin{equation}{\label{eq:qstarintro}}
		q_* := \begin{cases} 2\alpha Q^{*}+2-4\alpha, & \tfrac12<\alpha<1,\\[2pt] 2\alpha Q^{*}-2\alpha, & \alpha \geq 1.\end{cases}
	\end{equation}
	 For details, see Remark \ref{rem:qstar}.

	\begin{theorem}{\label{theorem:T}}
		Let $\max\{4\alpha,q_*\}<q<2\alpha Q^{*}$ and $p=2\alpha Q^{*}$. Assume $(V_1)$ together with $(V_2)$ or $(V_2^\prime)$. Then \eqref{eq:P} has a nontrivial solution.
	\end{theorem}

	For the second result, we impose a condition on the derivative of $V$ along the dilation orbits. Writing
	\begin{align*}
		Z= \sum_{i=1}^{N}\left(x_i \partial_{x_i} + y_i \partial_{y_i}\right) + 2t\,\partial_t
	\end{align*}
	for the generator of $\{\delta_\theta\}_{\theta>0}$, we assume
	\begin{description}
		\item[$(V_3)$] $V \in C^1(\H)$ and there is $\kappa \leq Q-\frac{Q-2}{2\alpha}$ with $ ZV(\xi) \geq -\kappa\,V(\xi)$ for all $\xi \in \H$.
	\end{description}
	The special case $\kappa=0$, that is $ZV \geq 0$, in which $V$ is nondecreasing along every dilation orbit, is the exact analogue of the hypothesis $\nabla V(x)\cdot x \geq 0$ of \cite{liu2004solutions}. 
    The value $Q-\frac{Q-2}{2\alpha}$ is the largest $\kappa$ for which the coefficient of $\int_{\mathbb{H}^N} V u^2\,d\xi$ in (7.2) is nonpositive; see the proof of Theorem 1.2.

	The nonexistence theorem is stated for a general nonlinearity subject to the growth conditions $(f_1)$ and $(f_2)$ of Section 3, which are what make the regularity theory of that section available.

	\begin{theorem}{\label{thm:nonex}}
		Let $\alpha>\frac12$ and let $V$ satisfy $(V_1)$ and $(V_3)$. Let $f \in C(\R)$ be odd, with primitive $F$, and assume $(f_1)$, $(f_2)$ together with
		\begin{align*}
			f(s)s \geq 2\alpha Q^{*}F(s), \quad \text{ for all } s \in \R.
		\end{align*}
		Then the equation
		\begin{equation}\tag{$P_f$}\label{eq:Pf}
			-\Delta_{\mathbb{H}}u+V(\xi)u-\Delta_{\mathbb{H}}\!\left(\left|u\right|^{2\alpha}\right)\left|u\right|^{2\alpha-2}u=f(u), \quad \xi \in \H,
		\end{equation}
		has no nontrivial weak solution $u=g(v)$ with $v \in S^2_1(\H)$.
	\end{theorem}

	\begin{co}{\label{cor:lambdaintro}}
		Assume $(V_1)$ and $(V_3)$, and let $2<q<2\alpha Q^{*}$ and $\lambda \leq 0$. Then
		\begin{align*}
			-\Delta_{\mathbb{H}} u + V(\xi)u - \Delta_{\mathbb{H}}\!\left(\left|u\right|^{2\alpha}\right)\left|u\right|^{2\alpha-2}u = \lambda \left|u\right|^{q-2}u + \left|u\right|^{2\alpha Q^{*}-2}u \quad \text{ in } \H
		\end{align*}
		has no nontrivial weak solution $u=g(v)$ with $v \in S^2_1(\H)$. In particular, the purely critical problem, $\lambda=0$, has no nontrivial solution.
	\end{co}
	Read together, Theorem \ref{theorem:T} and Corollary \ref{cor:lambdaintro} say that 	under $(V_1)$, $(V_2)$ and $(V_3)$, for $p=2\alpha Q^{*}$ and 	$\max\{4\alpha,\,q_*\}<q<2\alpha Q^{*}$, the solvability of \eqref{eq:P} in the class 	$u=g(v)$, $v \in S^2_1(\H)$, is decided by the sign of $\lambda$: a nontrivial 	solution exists when $\lambda>0$ and none exists when $\lambda \leq 0$. The subcritical perturbation is a necessary ingredient. The three hypotheses are compatible, for instance for 	$V(\xi)=V_\infty-a(1+\left|\xi\right|_{\mathbb{H}}^{4})^{-1}$ with $0<a<V_\infty$; on 	the other hand $(V_1)$ and $(V_2)$ do not imply $(V_3)$, and a nonconstant 	$\Ga_{\mathbb{Z}}$-invariant potential satisfies neither, so Corollary 	\ref{cor:lambdaintro} and the case $(V_2^\prime)$ of Theorem \ref{theorem:T} do not 	overlap. See Remark \ref{rem:V4}.\\
    The proof of Theorem \ref{theorem:T} follows the dual approach of Colin and Jeanjean 	\cite{colin2004solutions}. We set $u=g(v)$, where $g$ is defined by an ordinary differential equation, which converts \eqref{eq:P} into a semilinear equation on 	$S^2_1(\H)$. The transformed functional has mountain pass geometry, and its Cerami 	sequences are bounded when $q>4\alpha$. Compactness fails, since the embedding 	$S^2_1(\H) \hookrightarrow L^{Q^{*}}(\H)$ is not compact, and we recover it in two steps. First, an energy estimate places the minimax level below $S^{Q/2}/(2\alpha Q)$; it is obtained by testing with truncated extremals of the Folland--Stein inequality \cite{citti1995semilinear}. Second, a vanishing lemma of 	Lions type \cite{lions1984concentration}, adapted to gauge balls, supplies a sequence 	$\{\xi_n\}$ along which the mass of $v_n$ does not escape. Translating by 	$\tau_{\xi_n}$ then yields a nontrivial weak limit.
	
	Theorem \ref{thm:nonex} rests instead on a Poho\v{z}aev identity for 	$\Delta_{\mathbb{H}}$, proved by testing with $\chi_R\,Zv$. There the commutator 	relations $[X_i,Z]=X_i$ and $[Y_i,Z]=Y_i$ replace the Euclidean scaling computation. 	The non-horizontal part $2tT$ of $Z$ is controlled by the exponential decay of Lemma 	\ref{lemma:decay}.

	The paper is organised as follows. Section 2 reviews the Heisenberg group and sets up the variational framework for \eqref{eq:P} after the change of variable. Section 3 proves that every weak solution is bounded and decays exponentially in the gauge, and uses this to establish the Poho\v{z}aev identity. Section 4 contains the technical lemmas: integrability of $g(v)$, the decomposition of $I_\lambda$ that isolates the critical term, the mountain pass geometry, and the boundedness of Cerami sequences. Section 5 gives the upper bound on the minimax level, Section 6 proves Theorem \ref{theorem:T}, and Section 7 proves Theorem \ref{thm:nonex} and Corollary \ref{cor:lambda}.

	\section{Preliminaries and Variational framework}

	This Section starts with a brief review of the basic properties of the Heisenberg group. For a complete treatment of the Heisenberg group, we refer to \cite{garofalo1990frequency, ivanov2011extremals} and references therein.
	The Heisenberg group $\H$ is a Lie group of topological dimension $2N + 1$, whose underlying manifold is $\R^{2N+1}$, endowed with the following group law:
	\begin{align*}
		\xi \circ \xi'=(x,y,t) \circ(x',y',t')=\left(x+x',\,y+y',\,t+t'+ 2\big(\langle x',y \rangle - \langle x,y' \rangle\big)\right),
	\end{align*}
	where $\xi=(x,y,t), \xi'= (x',y',t') \in \H.$

	The inverse of $\xi \in \H$ is given by $\xi^{-1}=-\xi$ and $(\xi \circ \xi')^{-1}=(\xi')^{-1} \circ \xi^{-1}.$ The corresponding Lie algebra of left invariant vector fields is generated by vector fields given as follows
	\begin{align*}
		X_i= \frac{\partial}{\partial x_i}+2y_i \frac{\partial}{\partial t}, \quad Y_i= \frac{\partial}{\partial y_i}- 2x_i \frac{\partial}{\partial t}, \quad T= \frac{\partial}{\partial t}.
	\end{align*}
	A direct computation gives, for all $i,j=1,2,\dots,N,$
	\begin{align*}
		[X_i,X_j]=[Y_i,Y_j]=[X_i,T]=[Y_i,T]=0
	\end{align*}
	and
	\begin{align*}
		[X_i,Y_j]= -4\delta_{ij} \frac{\partial}{\partial t}.
	\end{align*}
	These relations reproduce the canonical commutation relations of quantum mechanics for position and momentum, which is the origin of the name \cite{heisenberg2013physical}.
	We define the left translation $\tau_\xi: \H \to \H $ by $\tau_\xi(\xi')=\xi \circ \xi'$ and the anisotropic dilation $\delta_\theta: \H \to \H$ by
	\begin{align*}
		\delta_\theta(x,y,t)=(\theta x, \theta y, \theta^2t), \qquad \theta>0.
	\end{align*}
	The subelliptic Laplacian or Kohn Laplacian $\Delta_{\mathbb{H}}$ on $\H$ is a second order self-adjoint operator defined as follows:
	\begin{align*}
		\Delta_{\mathbb{H}}= \sum_{j=1}^{N}\left(X_{j}^2 +Y_{j}^2\right).
	\end{align*}
	H{\"o}rmander \cite{hormander1967hypoelliptic} studied operators given by sums of squares of vector fields, and that work is the starting point of the analysis on homogeneous Lie groups used here. The operator $\Delta_{\mathbb{H}}$ is hypoelliptic, and its fundamental solution was found by Folland \cite{folland1973fundamental}. In \cite{folland1975subelliptic}, the same author proved subelliptic estimates and introduced the associated function spaces on nilpotent Lie groups. In the divergence form, the Kohn Laplacian is defined as $\Delta_{\mathbb{H}} u= \nabla\cdot(B\nabla u)$, where $B$ is the following $(2N+1) \times (2N+1)$ matrix:
	\begin{align*}
		B=\begin{bmatrix}
			I & 0 & 2y^T \\
			0 & I & -2x^T \\
			2y & -2x & 4(x^2+y^2)
		\end{bmatrix}
	\end{align*}
	where $I$ is $N \times N$ identity matrix and $x^2+y^2=\sum_{j=1}^{N} x_j^2+y_j^2.$ Hence, the following Gauss-Green formula holds:
	\begin{align*}
		\int_\Omega \Delta_{\mathbb{H}} u\, v\,d\xi=-\int_\Omega \nah u\cdot \nah v \,d\xi + \int_{\partial \Omega} v\, B\nabla u\cdot\nu \, d\sigma,
	\end{align*}
	where $\nah u= (X_1u, \dots ,X_{N}u, Y_1u,\dots , Y_{N}u)$ is a $2N$ vector and $\nu$ is unit outward normal to the boundary $\partial \Omega.$
	The operator $\Delta_{\mathbb{H}}$ is invariant under left translations and homogeneous of degree $2$ with respect to $\de_\theta$; that is,
	\begin{align*}
		\Delta_{\mathbb{H}}(u \circ \tau_\xi)=(\Delta_{\mathbb{H}} u) \circ \tau_\xi, \quad \Delta_{\mathbb{H}}(u \circ \delta_\theta)= \theta^2 (\Delta_{\mathbb{H}} u)\circ \delta_\theta.
	\end{align*}
	The Jacobian determinant of $\delta_\theta$ is $\theta^Q$. The number $Q=2N+2$ is called the homogeneous dimension of $\H$, and it plays a role equivalent to topological dimension in Euclidean space. The homogeneous norm on $\H$ is defined by
	\begin{align*}
		\left|\xi\right|_\mathbb{H}=\left|(x,y,t)\right|_\mathbb{H}=(t^2+(x^2+y^2)^2)^\frac{1}{4} \quad \text{ for all } \xi=(x,y,t) \in \H,
	\end{align*}
	and $B_r(\xi_0)=\{\xi \in \H: \left|\xi_0^{-1}\circ \xi\right|_\mathbb{H}<r\}$ denotes the corresponding gauge ball.

	Analogous to space $W^{1,2}(\mathbb{R}^m)$, Folland and Stein \cite{folland1974estimates} introduced the space $S^2_1(\H)$ which is related to vector fields $X_j$ and $Y_j.$ The space
	\begin{align*}
		S^2_1(\H)= \{u\in L^2(\H): X_j u, Y_j u \in L^2(\H), \text{ for all } j=1,2,\dots,N\}
	\end{align*}
	is a Hilbert space with inner product
	\begin{align*}
		\langle u,v\rangle= \int_\H \nah u\cdot \nah v\, d\xi + \int_\H uv\, d\xi
	\end{align*}
	and the corresponding norm is
	\begin{align*}
		\|u\|_{S^2_1(\H)}= \left(\int_\H \left|\nah u\right|^2 d\xi + \int_\H \left|u\right|^2 d\xi\right)^\frac{1}{2}.
	\end{align*}

	In \cite{folland1974estimates}, Folland and Stein proved the following Sobolev-type inequality: there exists a positive constant $C_S$ such that
	\begin{align*}
		\left|u\right|_{Q^*} \leq C_S \|\nah u\|_{2}, \quad  \text{for all } u \in S^2_1(\H),
	\end{align*}
	where $\left|u\right|_{Q^*}$ is the norm in $L^{Q^*}(\H).$ Interpolating between $L^2(\H)$ and $L^{Q^*}(\H)$ gives the continuous embedding $S^2_1(\H) \hookrightarrow L^m(\H)$ for every $m \in [2,Q^*]$. The best constant for the embedding $S^2_1(\H)$ into $L^{Q^*}(\H)$ is defined as
	\begin{align*}
		S= \inf\left\{\int_\H \left| \nah u \right|^2 d\xi \;:\; u \in S^2_1(\H),\; \int_\H \left|u \right|^{Q^*} d\xi=1\right\}.
	\end{align*}

	The natural energy functional associated with problem \eqref{eq:P} is defined as
	\begin{align*}
		J_{\lambda}(u)= \frac{1}{2} \int_\H \left(1+ 2\alpha \left|u\right|^{2(2\alpha-1)}\right) \left|\nah u\right|^2\ d\xi +\frac{1}{2} \int_\H V(\xi) u^2d\xi- \int_\H F(u)d\xi,
	\end{align*}
	where $F(u)= \frac{\lambda}{q} \left|u\right|^{q} +\frac{1}{p} \left|u\right|^{p}$ is the primitive of $f(u)= \lambda \left|u\right|^{q-2}u + \left|u\right|^{p-2}u.$ This functional is not well defined on $S^2_1(\H)$ because of the term $\left|u\right|^{2(2\alpha-1)}\left|\nah u\right|^2$. To deal with it, we employ a change of variable $u=g(v),$ where $g$ is the solution of the initial value problem
	\begin{align*}
		g^{\prime}(s)&=\left(1+2\alpha \left|g(s)\right|^{2(2\alpha -1)}\right)^{- \frac{1}{2}} \text{ on } [0,\infty), \\g(s)&=-g(-s) \text{ on } (-\infty,0].
	\end{align*}
	After this transformation, the functional $I_{\lambda}: S^2_1(\H) \to \R$ is
	\begin{align*}
		I_{\lambda}(v)= J_{\lambda}(g(v))= \frac{1}{2} \int_\H \left|\nah v\right|^2d\xi + \frac{1}{2} \int_\H V(\xi)g^2(v)d\xi- \int_\H F(g(v))d\xi.
	\end{align*}
	The finiteness of $I_\lambda$ on $S^2_1(\mathbb{H}^N)$ is proved in section 4 ; the proof uses $\alpha>\frac12$ and the Folland--Stein embedding. It follows that $\I \in C^1(S^2_1(\H), \R)$ with
	\begin{align*}
		\langle \I^{\prime}(v),w \rangle= \int_\H \nah v\cdot \nah w\, d\xi + \int_\H [V(\xi)g(v)-f(g(v))] g^{\prime}(v) w\, d\xi,
	\end{align*}
	for every $w \in  S^2_1(\H).$
	Critical points of $I_{\lambda}$ are weak solutions of
	\begin{equation}{\label{eq:P1}}
		-\Delta_{\mathbb{H}} v= g^{\prime}(v)\left[f(g(v))-V(\xi)g(v)\right] \quad \text{ in } \H,
	\end{equation}
	and $u=g(v)$ then solves \eqref{eq:P}.

	We record the properties of $g$ used throughout the article. For the proofs, we refer to \cite{colin2004solutions, li2015positive}.
	\begin{lemma}{\label{lemma:l1}}
		The transformation $``g"$ possesses the following properties:
		\begin{description}
			\item[$(g_1)$] $g$ is uniquely defined and invertible. Also, $g \in C^{\infty}(\R);$
			\item[$(g_2)$] $g(0)=0$, $g^\prime(0)=1,$ and  $0<g^{\prime}(t) \leq 1, \text{ for all } t\in \R ;$
			\item[$(g_3)$] $ \begin{cases} \frac{1}{2} g(t)\leq \alpha t g^{\prime}(t) \leq \alpha g(t), & t>0, \\ \frac{1}{2} g(t)\geq \alpha t g^{\prime}(t) \geq \alpha g(t), & t<0; \end{cases}$
			\item[$(g_4)$] $\left|g(t)\right| \leq \left|t\right|$ and $\left|g(t)\right|^{2\alpha} \leq (2\alpha)^{\frac{1}{2}} \left|t\right|, \text{ for all } t\in \R;$

			\item[$(g_5)$] $\displaystyle \lim_ {t\to \infty}\frac{g(t)}{t^{\frac{1}{2\alpha}}}=(2\alpha)^{\frac{1}{4\alpha}};$ $\left|g(t)\right|^{2\alpha -1} g^{\prime}(t) \leq \frac{1}{\sqrt{(2\alpha)}}, \text{ for all } t\in \R;$
			\item[$(g_6)$] $ \left|g(t)\right| \geq \begin{cases}  g(1)\left|t\right|, & \left|t\right| \leq 1, \\  g(1)\left|t\right|^{\frac{1}{2\alpha}}, &\left|t\right| \geq 1; \end{cases} $
			\item[$(g_7)$] $g^{\prime \prime}(t) <0, \text{ for } t>0 \text{ and } g^{\prime \prime}(t) >0, \text{ for } t<0;$
			\item[$(g_8)$] the function $t \mapsto \dfrac{g(t)^{\ba}\,g^{\prime}(t)}{t}$ is strictly increasing on $(0,\infty)$ for every $\ba \geq 4\alpha-1$.
		\end{description}
	\end{lemma}

	\section{Regularity, decay and the Poho\v{z}aev identity}

	This section proves that every weak solution of \eqref{eq:P1} is bounded and decays exponentially in the gauge, and uses this to establish the Poho\v{z}aev identity. The results are used twice: in Section 6, for the limit problem, and in Section 7, for the nonexistence theorem. For Lemmas \ref{lemma:linf} and \ref{lemma:decay}, the potential is required to satisfy $(V_1)$ along with $(V_2)$, so a constant potential is covered; Proposition \ref{prop:poho} needs $(V_3)$ in addition.

	We write $\varrho(\xi)=\left|\xi\right|_{\mathbb{H}}$, $\zeta=(x,y)$ and
	\begin{align*}
		A_R = \left\{\xi \in \H \;:\; R< \varrho(\xi) <2R\right\}, R>0,
	\end{align*}
	for the gauge annulus. We fix once and for all $\chi \in C^\infty_c([0,\infty))$ with $0 \leq \chi \leq 1$, $\chi \equiv 1$ on $[0,1]$ and $\chi \equiv 0$ on $[2,\infty)$, and put $\chi_R(\xi)=\chi(\varrho(\xi)/R)$, so that $\nah \chi_R$ and $Z\chi_R$ are supported in $A_R$.

	The nonlinearity is assumed only to satisfy
	\begin{description}
		\item[$(f_1)$] there is $C_f>0$ with $\left|f(t)\right| \leq C_f\left(\left|t\right|+\left|t\right|^{2\alpha Q^{*}-1}\right)$ for all $t \in \R$;
		\item[$(f_2)$] there are $\ell<V_0$ and $\eta_0>0$ with $f(t)t \leq \ell\,t^2$ for all $\left|t\right| \leq \eta_0$.
	\end{description}
	Both hold for the nonlinearity of \eqref{eq:P} whenever $2<q<p \leq 2\alpha Q^{*}$ and $\lambda \in \R$: for $(f_2)$ one may take any $\ell \in (0,V_0)$, since $f(t)t=O(\left|t\right|^{q})$ near the origin and $q>2$.

	\begin{lemma}{\label{lemma:Zcomm}}
		The vector field $Z$ satisfies
		\begin{itemize}
			\item[(i)] $\operatorname{div} Z = Q$, where the divergence is taken in $\R^{2N+1}$; consequently, for $u \in C^1_0(\H)$ and $\varphi \in C^1(\H)$,
			\begin{align*}
				\int_\H \varphi\, Zu\, d\xi = -\int_\H u \left(Z\varphi + Q\varphi\right)d\xi;
			\end{align*}
			\item[(ii)] $[X_i, Z]= X_i$ and $[Y_i,Z]=Y_i$ for every $i=1,\dots,N$, and hence
			\begin{align*}
				\nah (Zu) = Z\left(\nah u\right) + \nah u,
			\end{align*}
			the field $Z$ acting componentwise;
			\item[(iii)] $ Z= \sum_{i=1}^{N}\left(x_i X_i + y_i Y_i\right) + 2tT$.
		\end{itemize}
	\end{lemma}
	\begin{proof}
		(i) Regarded as a vector field on $\R^{2N+1}$, $Z=(x,y,2t)$, so $\operatorname{div}Z = N + N + 2 = Q$. The integration by parts formula is the divergence theorem applied to $u\varphi Z$.

		(ii) Write $X_i = \partial_{x_i}+2y_i\partial_t$. Since $\partial_{x_i}(Zu)= u_{x_i}+Z(u_{x_i})$ and $\partial_t (Zu) = 2u_t + Z(u_t)$,
		\begin{align*}
			X_i(Zu) = u_{x_i}+Z(u_{x_i}) + 2y_i\left(2u_t + Z(u_t)\right),
		\end{align*}
		whereas
		\begin{align*}
			Z(X_i u) = Z(u_{x_i}) + 2Z(y_i u_t) = Z(u_{x_i}) + 2y_i u_t + 2y_i Z(u_t).
		\end{align*}
		Subtracting, $[X_i,Z]u = u_{x_i}+2y_i u_t = X_i u$. Similarly, the result holds for $Y_i = \partial_{y_i}-2x_i\partial_t$.

		(iii) Substituting the definitions of $X_i, Y_i$ into the right-hand side, the coefficients of $\partial_t$ contribute $\sum_i (2x_iy_i - 2y_ix_i)=0$. \qed
	\end{proof}
	\begin{lemma}{\label{lemma:linf}}
		Assume $(V_1)$ and $(V_2)$, together with $(f_1)$, and let $v \in S^2_1(\H)$ be a weak solution of \eqref{eq:P1}. Then $v \in L^{r}(\H)$ for every $r \in [2,\infty].$ In particular, $v \in L^\infty(\H)$.
	\end{lemma}
	\begin{proof}
		Write $k(\xi,s)=g^{\prime}(s)\left[f(g(s))-V(\xi)g(s)\right]$. By $(g_2)$ and $(g_4)$, we have $g^{\prime}(s)\left|g(s)\right| \leq \left|s\right|$, while $(g_5)$ and $(g_4)$ give
		\begin{align*}
			g^{\prime}(s)\left|g(s)\right|^{2\alpha Q^{*}-1} = g^{\prime}(s)\left|g(s)\right|^{2\alpha-1}\left(\left|g(s)\right|^{2\alpha}\right)^{Q^{*}-1} \leq (2\alpha)^{\frac{2}{Q-2}}\left|s\right|^{Q^{*}-1}.
		\end{align*}
		With $(f_1)$ and $V \leq V_M$, we obtain
		\begin{equation}{\label{eq:kbound}}
			\left|k(\xi,s)\right| \leq C\left(\left|s\right|+\left|s\right|^{Q^{*}-1}\right), \quad \xi \in \H, \ s \in \R .
		\end{equation}
		Fix $L>0$ and split $\left|k(\xi,v)\right| \leq C(1+\left|v\right|^{Q^*-2})\left|v\right| \leq \left(a_1+a_2\right)\left|v\right|$, where
		\begin{align*}
			a_1 = C\left|v\right|^{Q^{*}-2}\chi_{\{\left|v\right|>L\}}, \qquad a_2 = C\left(1+L^{Q^{*}-2}\right).
		\end{align*}
		Since $v \in L^{Q^{*}}(\H)$, we have $a_1 \in L^{\frac{Q}{2}}(\H)$, with $\left|a_1\right|_{Q/2}$ as small as we please once $L$ is large; and $a_2$ is a constant.

Fix $\ba \geq 1$ and $\tau>0$, set $v_\tau=\min\{\left|v\right|,\tau\}$ and
$w_\tau=\left|v_\tau\right|^{\ba-1}v$. The function
$\varphi=\left|v_\tau\right|^{2(\ba-1)}v$ is admissible in \eqref{eq:P1}:
$\left|\varphi\right| \leq \tau^{2(\ba-1)}\left|v\right|$ and, since
$\nah\varphi=(2\ba-1)\left|v\right|^{2(\ba-1)}\nah v$ on $\{\left|v\right|<\tau\}$ while
$\nah\varphi=\tau^{2(\ba-1)}\nah v$ on $\{\left|v\right|>\tau\}$, also
$\left|\nah\varphi\right| \leq (2\ba-1)\tau^{2(\ba-1)}\left|\nah v\right|$, so
$\varphi \in S^2_1(\H)$. The same computation gives
\begin{align*}
\int_\H \nah v\cdot\nah\varphi\, d\xi
\geq \int_\H \left|v_\tau\right|^{2(\ba-1)}\left|\nah v\right|^2d\xi
\geq \frac{1}{\ba^2}\int_\H \left|\nah w_\tau\right|^2d\xi,
\end{align*}
the last inequality from
$\left|\nah w_\tau\right| \leq \ba\left|v_\tau\right|^{\ba-1}\left|\nah v\right|$.
Testing \eqref{eq:P1} with $\varphi$ and applying the Folland--Stein inequality to
$w_\tau$,
\begin{equation}{\label{eq:moserbasic}}
\left|w_\tau\right|_{Q^{*}}^2 \leq C_S^2\int_\H \left|\nah w_\tau\right|^2 d\xi
\leq C_S^2\ba^2\int_\H \left|k(\xi,v)\right|\left|v_\tau\right|^{2(\ba-1)}\left|v\right|d\xi
\leq C_S^2\ba^2\int_\H (a_1+a_2)\,w_\tau^2\,d\xi .
\end{equation}

\emph{Step 1: $v \in L^{\frac{(Q^{*})^2}{2}}(\H)$.} Take $\ba=\frac{Q^{*}}{2}$. H\"older
with exponents $\frac{Q}{2}$ and $\frac{Q}{Q-2}$ gives
$\int_\H a_1w_\tau^2\,d\xi \leq \left|a_1\right|_{Q/2}\left|w_\tau\right|_{Q^{*}}^2$.
Choose $L$ so large that $C_S^2\ba^2\left|a_1\right|_{Q/2}\leq\frac12$; absorbing into
\eqref{eq:moserbasic},
\begin{align*}
\left|w_\tau\right|_{Q^{*}}^2 \leq 2C_S^2\ba^2a_2\int_\H w_\tau^2\,d\xi
\leq 2C_S^2\ba^2a_2\left|v\right|_{2\ba}^{2\ba},
\end{align*}
finite because $2\ba=Q^{*}$. Letting $\tau\to\infty$, monotone convergence yields
$v \in L^{\ba Q^{*}}(\H)=L^{\frac{(Q^{*})^2}{2}}(\H)$.

\emph{Step 2: iteration with constants polynomial in $\ba$.}
Put
$b=C\left|v\right|^{Q^{*}-2}$ and $s=\frac{(Q^{*})^2}{2(Q^{*}-2)}$, so that
$(Q^{*}-2)s=\frac{(Q^{*})^2}{2}$ and $b \in L^{s}(\H)$ by Step 1. Moreover
$s>\frac{Q}{2}$: substituting $Q^{*}=\frac{2Q}{Q-2}$, the inequality
$(Q^{*})^2-QQ^{*}+2Q>0$ reduces to $\frac{8Q}{(Q-2)^2}>0$. Hence
$2s^{\prime}\in(2,Q^{*})$, where $s^{\prime}=\frac{s}{s-1}$, and there is
$\theta\in(0,1)$ with $\frac{1}{2s^{\prime}}=\frac{1-\theta}{2}+\frac{\theta}{Q^{*}}$.
From \eqref{eq:kbound}, $\left|k(\xi,v)\right| \leq C\left|v\right|+b\left|v\right|$,
so \eqref{eq:moserbasic} gives, for every $\ba\geq1$,
\begin{align*}
\left|w_\tau\right|_{Q^{*}}^2
\leq C_S^2\ba^2\left(C\left|w_\tau\right|_2^2
+\left|b\right|_{s}\left|w_\tau\right|_{2s^{\prime}}^2\right).
\end{align*}
By interpolation and Young's inequality, for every $\e>0$,
\begin{align*}
\left|w_\tau\right|_{2s^{\prime}}^2
\leq \left|w_\tau\right|_2^{2(1-\theta)}\left|w_\tau\right|_{Q^{*}}^{2\theta}
\leq \e\left|w_\tau\right|_{Q^{*}}^2
+C_\theta\,\e^{-\frac{\theta}{1-\theta}}\left|w_\tau\right|_2^2 .
\end{align*}
Choosing $\e=(2C_S^2\ba^2\left|b\right|_s)^{-1}$ and absorbing,
\begin{align*}
\left|w_\tau\right|_{Q^{*}}^2 \leq C_1\ba^{m}\left|w_\tau\right|_2^2,
\qquad m:=2+\frac{2\theta}{1-\theta},
\end{align*}
with $C_1$ depending only on $Q$, $C_S$, the constant of \eqref{eq:kbound} and
$\left|b\right|_s$, not on $\ba$ or $\tau$. Letting $\tau\to\infty$,
\begin{equation}{\label{eq:moseriter}}
\left|v\right|_{\ba Q^{*}} \leq \left(C_1\ba^{m}\right)^{\frac{1}{2\ba}}
\left|v\right|_{2\ba} \quad \text{ whenever } v \in L^{2\ba}(\H).
\end{equation}

Iterating \eqref{eq:moseriter} over $\ba_j=\left(\frac{Q^{*}}{2}\right)^{j}$, for
which $2\ba_{j+1}=\ba_jQ^{*}$ and $2\ba_1=Q^{*}$, the resulting product converges since
$\sum_j\ba_j^{-1}\left(\ln C_1+m\ln\ba_j\right)<\infty$, and letting $j\to\infty$ gives
$\left|v\right|_\infty \leq \La\left|v\right|_{Q^{*}}$ for some $\La<\infty$.
Interpolation with $v \in L^2(\H)$ completes the proof.
 \qed
	\end{proof}
\begin{co}{\label{cor:localmoser}}
Under the hypotheses of Lemma \ref{lemma:linf}, there exists $C>0$, independent of
$\xi_0 \in \H$, such that
$\left\|v\right\|_{L^\infty(B_1(\xi_0))} \leq C\left|v\right|_{L^{Q^{*}}(B_2(\xi_0))}$.
\end{co}
\proof For $1\leq\rho^{\prime}<\rho\leq2$, pick $\eta \in C_c^\infty(B_\rho(\xi_0))$
with $\eta\equiv1$ on $B_{\rho^{\prime}}(\xi_0)$, $0\leq\eta\leq1$ and
$\left|\nah\eta\right| \leq \frac{C}{\rho-\rho^{\prime}}$. Testing \eqref{eq:P1} with
$\eta^2\left|v_\tau\right|^{2(\ba-1)}v$ and arguing as in \eqref{eq:moserbasic}, the
cross term $2\int_\H \eta\left|v_\tau\right|^{2(\ba-1)}v\,\nah v\cdot\nah\eta\, d\xi$
being handled by Young's inequality,
\begin{align*}
\int_\H \left|\nah(\eta w_\tau)\right|^2d\xi
\leq C\ba^2\left(\int_\H \left|\nah\eta\right|^2w_\tau^2\,d\xi
+\int_\H \eta^2(1+b)\,w_\tau^2\,d\xi\right),
\end{align*}
with $b=C\left|v\right|^{Q^{*}-2}$. The Folland--Stein inequality applied to
$\eta w_\tau$, H\"older with the exponent $s$ of Step 2 of Lemma \ref{lemma:linf} and
the same interpolation--absorption argument, now for $\eta w_\tau$, give, after
$\tau\to\infty$,
\begin{align*}
\left|v\right|_{L^{\ba Q^{*}}(B_{\rho^{\prime}}(\xi_0))}
\leq \left(\frac{C_1\ba^{m}}{(\rho-\rho^{\prime})^{2}}\right)^{\frac{1}{2\ba}}
\left|v\right|_{L^{2\ba}(B_{\rho}(\xi_0))},
\end{align*}
where $C_1$ depends only on $Q$, $C_S$, the constant of \eqref{eq:kbound} and
$\left|b\right|_{L^{s}(\H)}$, the latter finite by Lemma \ref{lemma:linf} and
independent of $\xi_0$. Iterating over $\ba_j=(\frac{Q^{*}}{2})^{j}$ and
$\rho_j=1+2^{-j+1}$, the product
$\prod_j\big(C_1\,4^{\,j}\ba_j^{m}\big)^{\frac{1}{2\ba_j}}$ converges, and the limit
argument of Lemma \ref{lemma:linf} yields the stated bound. Left translations preserve
$\nah$, the gauge balls and the Haar measure, so no constant depends on $\xi_0$. \qed
	\begin{lemma}{\label{lemma:decay}}
		Assume $(V_1)$ and $(V_2)$, together with $(f_1)$ and $(f_2)$, and let $v$ be as in Lemma \ref{lemma:linf}. Then $v \in \Ga^{2,\gamma}_{loc}(\H)$ for some $\gamma \in (0,1)$, and there exist $C,c>0$ with
		\begin{equation}{\label{eq:expdecay}}
			\left|v(\xi)\right| \leq C e^{-c\varrho(\xi)}, \qquad \xi \in \H.
		\end{equation}
		Moreover, for every $\xi_0 \in \H$,
		\begin{equation}{\label{eq:Tbound}}
			\left\|\nah v\right\|_{L^2(B_1(\xi_0))}+\left\|Tv\right\|_{L^2(B_1(\xi_0))}
			\leq C e^{-c\varrho(\xi_0)}.
		\end{equation}
	\end{lemma}
	\begin{proof}
		Since $v \in L^\infty(\H)$ and $s \mapsto k(\xi,s)$ is locally H\"older continuous, the subelliptic Schauder estimates \cite{folland1975subelliptic} give $v \in \Ga^{2,\gamma}_{loc}(\H)$. By Corollary \ref{cor:localmoser},
        		\begin{align*}
			\left\|v\right\|_{L^\infty(B_1(\xi_0))} \leq C \left|v\right|_{L^{Q^{*}}(B_2(\xi_0))},
		\end{align*}
		with $C$ independent of $\xi_0$; since $v \in L^{Q^{*}}(\H)$, the right-hand side tends to $0$ as $\varrho(\xi_0) \to \infty$, so $v(\xi) \to 0$ at infinity.

		We now compare with an exponential barrier. Put $\Psi=\left|\nah \varrho\right|^2$. A direct computation from $\varrho^4=t^2+\left|\zeta\right|^4$ gives
		\begin{align*}
			X_i\varrho=\frac{ty_i+\left|\zeta\right|^2x_i}{\varrho^3}, \qquad
			Y_i\varrho=\frac{\left|\zeta\right|^2y_i-tx_i}{\varrho^3}, \qquad
			\Psi=\frac{\left|\zeta\right|^2}{\varrho^2}\in[0,1],
		\end{align*}
		and hence, for $h \in C^2((0,\infty))$,
		\begin{equation}{\label{eq:radialLap}}
			\Delta_{\mathbb{H}}\left(h(\varrho)\right)
			=\Psi\left(h^{\prime\prime}(\varrho)+\frac{Q-1}{\varrho}h^{\prime}(\varrho)\right)
			\qquad \text{on } \H \setminus \{0\},
		\end{equation}
		which is consistent with $\Delta_{\mathbb{H}}\varrho^{2-Q}=0$.

		We claim that
		\begin{equation}{\label{eq:keta}}
			k(\xi,s)\,s \leq -\delta s^2 \quad \text{ for all } \xi \in \H \text{ and } \left|s\right| \leq \eta, 
		\end{equation}
		where $\delta:=\frac{(V_0-\ell)\,g(1)^2}{2\alpha}>0 \text{ and } \eta \in (0,1]$ is chosen with $\left|g(s)\right| \leq \eta_0$ for $\left|s\right| \leq \eta$. Indeed, let $0<s\leq\eta$. Then $g(s)>0$ and $(f_2)$ gives $f(g(s)) \leq \ell\, g(s)$, so by $(V_1)$,
		\begin{align*}
			f(g(s))-V(\xi)g(s) \leq (\ell-V_0)\,g(s)<0 .
		\end{align*}
		By $(g_3)$, we have $s\,g^{\prime}(s) \geq \frac{g(s)}{2\alpha}>0$, and by $(g_6)$, $g(s) \geq g(1)s$. Then,
		\begin{align*}
			k(\xi,s)s = s\,g^{\prime}(s)\left[f(g(s))-V(\xi)g(s)\right] \leq \frac{g(s)}{2\alpha}(\ell-V_0)g(s) \leq -\delta s^2 .
		\end{align*}
		For $-\eta \leq s<0$ the same arguments applies, since $g(s)<0$, the inequality $f(g(s)) \geq \ell\,g(s)$ from $(f_2)$ then gives $f(g(s))-V(\xi)g(s) \geq (\ell-V_0)g(s)>0$, and $s\,g^{\prime}(s) \leq \frac{g(s)}{2\alpha}<0$ by $(g_3)$. This proves \eqref{eq:keta}.

		As $v$ is continuous and vanishes at infinity, choose $\varrho_0>0$ with $\left|v\right| \leq \eta$ on $\Omega_0:=\{\varrho>\varrho_0\}$, and put
		\begin{align*}
			h(\varrho)=Me^{-c\varrho}, \quad c^2 \leq \delta, \quad M=\left\|v\right\|_\infty e^{c\varrho_0},
		\end{align*}
		so that $h \geq \left\|v\right\|_\infty$ on $\{\varrho \leq \varrho_0\}$. By \eqref{eq:radialLap},
		\begin{equation}{\label{eq:barrier}}
			-\Delta_{\mathbb{H}}h+\delta h = h\left(\delta-\Psi c^{2}+\Psi\,\frac{c(Q-1)}{\varrho}\right) \geq 0 \quad \text{on } \Omega_0,
		\end{equation}
		the inequality using only $0 \leq \Psi \leq 1$ and $c^2 \leq \delta$.

		Set $w=(v-h)^{+}$, extended by $0$ outside $\Omega_0$. Since $\left|v\right| \leq \left\|v\right\|_\infty \leq h$ on $\partial \Omega_0$, $w$ vanishes there. Moreover, $0 \leq w \leq \left|v\right|$ and $\nah w = \left(\nah v-\nah h\right)\chi_{\{v>h\}}$, so $w \in S^2_1(\H)$. On $\{w>0\}$, we have $0<v \leq \eta$, so \eqref{eq:keta} gives $k(\xi,v) \leq -\delta v$, that is, $-\Delta_{\mathbb{H}}v+\delta v \leq 0$ there. Subtracting \eqref{eq:barrier} and testing with $w$,
		\begin{align*}
			\int_{\H}\left|\nah w\right|^2 d\xi + \delta\int_{\H} w^2 d\xi \leq 0,
		\end{align*}
		so $w \equiv 0$, that is, $v \leq h$ on $\Omega_0$. The same argument applied to $(-v-h)^{+}$, using the second half of \eqref{eq:keta}, gives $-v \leq h$. Hence $\left|v\right| \leq h$ on $\Omega_0$, and the choice of $M$ extends this to all of $\H$, which is \eqref{eq:expdecay}.

		Finally, $\left|\Delta_{\mathbb{H}}v\right|=\left|k(\xi,v)\right| \leq C\left|v\right|$ by \eqref{eq:kbound} and $v \in L^\infty(\H)$, so the interior subelliptic estimate on $B_1(\xi_0) \Subset B_2(\xi_0)$ \cite{folland1975subelliptic},
		\begin{align*}
			\left\|\nah v\right\|_{L^2(B_1(\xi_0))}+\left\|Tv\right\|_{L^2(B_1(\xi_0))}
			\leq C\left(\left\|\Delta_{\mathbb{H}}v\right\|_{L^2(B_2(\xi_0))}
			+\left\|v\right\|_{L^2(B_2(\xi_0))}\right),
		\end{align*}
		together with \eqref{eq:expdecay}, yields \eqref{eq:Tbound}. \qed
	\end{proof}

	\begin{co}{\label{cor:Dauto}}
		Under the hypotheses of Lemma \ref{lemma:decay},
		\begin{align*}
			\lim_{R \to \infty} R\int_{A_R}\left|Tv\right|\left|\nah v\right|d\xi = 0 .
		\end{align*}
	\end{co}
	\begin{proof}
		Cover $A_R$ by $C R^{Q}$ gauge balls of radius $1$ with centres in $\{\varrho > R\}$ and bounded overlap. By Cauchy--Schwarz on each ball and \eqref{eq:Tbound},
		\begin{align*}
			R\int_{A_R}\left|Tv\right|\left|\nah v\right|d\xi
			\leq C R^{Q+1}e^{-2cR} \longrightarrow 0 \text{ as } R\to \infty. \qed
		\end{align*}
	\end{proof}

	\begin{prop}{\label{prop:poho}}
		Assume $(V_1)$ and $(V_3)$, together with $(f_1)$ and $(f_2)$, and let $u=g(v)$ with $v \in S^2_1(\H)$ a weak solution of \eqref{eq:P1}. Then
		\begin{equation}{\label{eq:poho}}
			\frac{Q-2}{2}\int_\H \left(1+2\alpha \left|u\right|^{4\alpha-2}\right)\left|\nah u\right|^2 d\xi = Q\int_\H \left(F(u)-\frac{1}{2}V(\xi)u^2\right)d\xi - \frac{1}{2}\int_\H \left(ZV\right)u^2 d\xi.
		\end{equation}
	\end{prop}
	\begin{proof}
		Recall from Section 2 that $v$ solves $-\Delta_{\mathbb{H}}v = k(\xi,v)$ with
		\begin{align*}
			k(\xi,s)=g^\prime(s)\left[f(g(s))-V(\xi)g(s)\right], \quad K(\xi,s)= \int_0^s k(\xi,\tau)d\tau = F(g(s))-\frac{1}{2}V(\xi)g^2(s),
		\end{align*}
		and that $\left|\nah v\right|^2 = \left(1+2\alpha\left|u\right|^{4\alpha-2}\right)\left|\nah u\right|^2$. Since the gauge is $\delta_\theta$-homogeneous of degree one, $Z\varrho=\varrho$, so $\left\|Z\chi_R\right\|_\infty \leq 2\left\|\chi^\prime\right\|_\infty$ uniformly in $A_R$; and $\left|\nah \varrho\right| \leq 1$ gives $\left|\nah \chi_R\right| \leq C/R$.

		Testing the equation with $\chi_R\, Zv$ and expanding,
		\begin{align*}
			\int_\H \chi_R\, \nah v \cdot \nah (Zv)\,d\xi + \int_\H (Zv)\,\nah v \cdot \nah \chi_R \,d\xi = \int_\H \chi_R\, k(\xi,v)\,Zv\, d\xi.
		\end{align*}
		By Lemma~\ref{lemma:Zcomm}(ii),
		\begin{align*}
			\nah v \cdot \nah (Zv) = \frac{1}{2}Z\left(\left|\nah v\right|^2\right)+\left|\nah v\right|^2,
		\end{align*}
		and Lemma~\ref{lemma:Zcomm}(i) turns the first integral into
		\begin{align*}
			-\frac{Q-2}{2}\int_\H \chi_R \left|\nah v\right|^2 d\xi - \frac{1}{2}\int_\H \left|\nah v\right|^2 Z\chi_R\, d\xi.
		\end{align*}
		On the right-hand side, $Z\left[K(\xi,v)\right]= -\frac{1}{2}(ZV)g^2(v)+k(\xi,v)Zv$, so again by Lemma~\ref{lemma:Zcomm}(i),
		\begin{align*}
			\int_\H \chi_R\, k(\xi,v)Zv\,d\xi = -Q\int_\H \chi_R K(\xi,v)d\xi - \int_\H K(\xi,v)Z\chi_R\,d\xi + \frac{1}{2}\int_\H \chi_R (ZV)g^2(v)d\xi.
		\end{align*}
		Since $v \in S^2_1(\H)$ we have $\left|\nah v\right|^2 \in L^1(\H)$, and $K(\xi,v) \in L^1(\H)$ by Lemma \ref{lemma:linf} together with the exponential decay of Lemma \ref{lemma:decay}. As $Z\chi_R$ is uniformly bounded and supported in $A_R$, the two terms carrying $Z\chi_R$ tend to $0$ as $R \to \infty$.

	At the end, consider,  $\mathcal{E}_R = \int_\H (Zv)\,\nah v \cdot \nah \chi_R \,d\xi$ and split $Z$ as in Lemma~\ref{lemma:Zcomm}(iii). On $A_R$ we have $\left|\zeta\right| \leq \varrho \leq 2R$, so the horizontal part is bounded by
		\begin{align*}
			C\,R \cdot \frac{C}{R}\int_{A_R}\left|\nah v\right|^2 d\xi = C\int_{A_R}\left|\nah v\right|^2d\xi \longrightarrow 0 \text{ as } R\to \infty.
		\end{align*}
		For the central part, $\left|t\right| \leq \varrho^2 \leq 4R^2$ gives the bound $CR\int_{A_R}\left|Tv\right|\left|\nah v\right|d\xi$, which tends to $0$ as $R\to \infty$ by Corollary \ref{cor:Dauto}. Letting $R \to \infty$,
		\begin{align*}
			\frac{Q-2}{2}\int_\H \left|\nah v\right|^2 d\xi = Q\int_\H K(\xi,v)d\xi - \frac{1}{2}\int_\H (ZV)g^2(v)d\xi,
		\end{align*}
		which is \eqref{eq:poho} after substituting $u=g(v)$. \qed
	\end{proof}

	\begin{lemma}{\label{lemma:nehari}}
		Under the hypotheses of Proposition~\ref{prop:poho}, 
		\begin{equation}{\label{eq:nehari}}
			\int_\H \left|\nah u\right|^2 d\xi + 4\alpha^2 \int_\H \left|u\right|^{4\alpha-2}\left|\nah u\right|^2 d\xi + \int_\H V(\xi)u^2 d\xi = \int_\H f(u)u\, d\xi.
		\end{equation}
	\end{lemma}
    \begin{proof}
		Testing \eqref{eq:Pf} with $u$, the quasilinear term integrates by parts as
		\begin{align*}
			-\int_\H \Delta_{\mathbb{H}}\left(\left|u\right|^{2\alpha}\right)\left|u\right|^{2\alpha-2}u\cdot u \,d\xi = \int_\H \left|\nah \left(\left|u\right|^{2\alpha}\right)\right|^2 d\xi = 4\alpha^2 \int_\H \left|u\right|^{4\alpha-2}\left|\nah u\right|^2 d\xi,
		\end{align*}
		and $u$ is an admissible test function by Lemma \ref{lemma:linf} and Lemma \ref{lemma:decay}. \qed
	\end{proof}

	\section{Technical lemmas}
	This section collects the material used in the proof of Theorem \ref{theorem:T}: an integrability lemma for $g(v)$, the decomposition of $I_\lambda$ that isolates the critical term, and the mountain pass geometry.

	\begin{lemma}{\label{lemma:lg}}
		Let $v \in S^2_1(\H)$ and set $w=\left|g(v)\right|^{2\alpha}$. Then $w \in S^2_1(\H)$ with $\|w\|\leq (2\alpha)^{\frac{1}{2}}\|v\|$, and
		\begin{align*}
			\int_\H \left|g(v)\right|^{r} d\xi \leq C\|v\|^{\frac{r}{2\alpha}}, \quad \text{ for every } r \in [4\alpha, 2\alpha Q^{*}].
		\end{align*}
	\end{lemma}
	\begin{proof}
		Since $\alpha>\frac{1}{2}$, the map $s \mapsto \left|s\right|^{2\alpha}$ is of class $C^1$ and the chain rule applies. Using $(g_5)$, we have
		\begin{align*}
			\left|\nah(\left|g(v)\right|^{2\alpha})\right|^2=&\left|2\alpha \left|g(v)\right|^{2\alpha-2} g(v) g^{\prime}(v) \nah v \right|^2 \\ =& 4\alpha^2 \left(\left|g(v)\right|^{2\alpha-1} g^{\prime}(v)\right)^2 \left|\nah v\right|^2 \\ \leq & 2\alpha \left|\nah v\right|^2.
		\end{align*}
		By $(g_4)$, $\left|g(v)\right|^{4\alpha}\leq 2\alpha \left|v\right|^2$, so $\left|w\right|_2^2 \leq 2\alpha \left|v\right|_2^2$ and consequently $\|w\|^2 \leq 2\alpha \|v\|^2$. Let $r \in [4\alpha, 2\alpha Q^{*}]$ and put $m=\frac{r}{2\alpha}\in [2,Q^{*}]$. The Folland-Stein embedding $S^2_1(\H)\hookrightarrow L^m(\H)$ gives
		\begin{align*}
			\int_\H \left|g(v)\right|^{r}d\xi = \int_\H w^{m}d\xi = \left|w\right|_m^m \leq C\|w\|^m \leq C \|v\|^{\frac{r}{2\alpha}}. \qed
		\end{align*}
	\end{proof}

	Since $p=2\alpha Q^{*}$, we may split off the critical term and rewrite $I_\lambda$ as
	\begin{equation}{\label{eq:I3}}
		\I(v)= \frac{1}{2} \int_\H \left(\left|\nah v\right|^2 + V(\xi)v^2\right)d\xi- \frac{(2\alpha)^{\frac{2}{Q-2}}}{ Q^{*}} \int_\H \left|v\right|^{Q^{*}} d\xi- \int_\H H(\xi,v)d\xi,
	\end{equation}
	where
	\begin{align*}
		H(\xi,v)= \frac{\lambda}{q} \left|g(v)\right|^{q} + \frac{1}{2\alpha Q^{*}} \left|g(v)\right|^{2\alpha Q^{*}}+ \frac{1}{2}V(\xi)v^2- \frac{1}{2} V(\xi)g^2(v)- \frac{(2\alpha)^{\frac{2}{Q-2}}}{Q^{*}} \left|v\right|^{Q^{*}}
	\end{align*}
	is the primitive in the second variable of
	\begin{align*}
		h(\xi,v)=g^{\prime}(v)\left(\lambda \left|g(v)\right|^{q-2}g(v) + \left|g(v)\right|^{2\alpha Q^{*}-2}g(v)-V(\xi)g(v)\right) +V(\xi)v- (2\alpha)^{\frac{2}{Q-2}}\left|v\right|^{Q^{*}-2}v.
	\end{align*}

	\begin{lemma}{\label{lemma:lH}}
		The functions $H$ and $h$ have the following properties, uniformly in $\xi \in \H$:
		\begin{itemize}
			\item[$(H_1)$]  $\lim\limits_{t \to 0} \frac{H(\xi,t)}{t^2}=0;$
			\item[$(H_2)$] $\lim\limits_{t \to \infty} \frac{H(\xi,t)}{t^{Q^{*}}}=0$;
			\item[$(H_3)$] $\lim\limits_{t \to 0} \frac{h(\xi,t)}{t}=0;$
			\item[$(H_4)$] $\lim\limits_{t \to \infty} \frac{h(\xi,t)}{t^{Q^{*}-1}}=0$.
		\end{itemize}
	\end{lemma}
	\begin{proof}
		To prove $(H_1)$, observe that
		\begin{align*}
			\frac{\left|g(t)\right|^{q}}{t^2}=\left(\frac{g(t)}{t}\right)^2 \left|g(t)\right|^{q-2}, \quad \frac{\left|g(t)\right|^{2\alpha Q^{*}}}{t^2}=\left(\frac{g(t)}{t}\right)^2 \left|g(t)\right|^{2\alpha Q^{*}-2}.
		\end{align*}
		Since $\displaystyle\lim_{t\to0} \frac{g(t)}{t}=1$ with $2\alpha Q^{*}>Q^{*}>2$ and $q>4\alpha>2$, both quotients tend to $0$. For the remaining terms, $0\le t^2-g^2(t) \leq t^2$ together with $g^\prime(0)=1$ gives $V(\xi)(t^2-g^2(t))/t^2 \to 0$ uniformly in $\xi$, since $V \leq V_M$, and $\left|t\right|^{Q^{*}}/t^2 \to 0$ because $Q^{*}>2$. Hence $\lim\limits_{t \to 0} \frac{H(\xi,t)}{t^2}=0.$

		To prove $(H_2)$, taking into account $(g_4),(g_5)$ and $q< 2\alpha Q^{*},$ we deduce
		\begin{align*}
			\frac{\left|g(t)\right|^{q}}{\left|t\right|^{Q^{*}}}={\left(\frac{\left|g(t)\right|}{\left|t\right|^{\frac{1}{2\alpha}}}\right)^{q} \left|t\right|^{\left(\frac{q}{2\alpha}-Q^{*}\right)}} \to 0 \text{ as } t\to \infty,
		\end{align*}
		\begin{align*}
			0 \leq \frac{1}{2}V(\xi)\frac{t^2}{\left|t\right|^{Q^{*}}}- \frac{1}{2}V(\xi)\frac{\left|g(t)\right|^2 }{\left|t\right|^{Q^{*}}} \leq \frac{1}{2} V_M \frac{t^2}{\left|t\right|^{Q^{*}}} \to 0 \text{ as } t\to \infty,
		\end{align*}
		and
		\begin{align*}
			\frac{\left|g(t)\right|^{2\alpha Q^{*}}}{\left|t\right|^{Q^{*}}}= \left(\frac{\left|g(t)\right|}{\left|t\right|^{\frac{1}{2\alpha}}}\right)^{2\alpha Q^{*}} \to \left((2\alpha)^\frac{1}{4\alpha}\right)^{2\alpha Q^{*}}= (2\alpha)^\frac{Q}{Q-2} \text{ as } t\to \infty.
		\end{align*}
		Since $\frac{Q}{Q-2}-1=\frac{2}{Q-2}$, the last limit gives $\frac{1}{2\alpha Q^{*}}(2\alpha)^{\frac{Q}{Q-2}}=\frac{(2\alpha)^{\frac{2}{Q-2}}}{Q^{*}}$, so the two critical terms in $H$ cancel in the limit. Thus, $\lim\limits_{t \to \infty} \frac{H(\xi,t)}{t^{Q^{*}}}=0.$

		The proof of $(H_3)$ follows the same lines as $(H_1)$. To prove $(H_4)$, thanks to the definition of $g$, we have
		\begin{align*}
			g^{\prime}(t) g(t) \frac{\left|g(t)\right|^{2\alpha Q^{*}-2}}{\left|t\right|^{Q^{*}-1}}&= g^{\prime}(t) g(t)\left|g(t)\right|^{2\alpha-2} \left(\frac{\left|g(t)\right|}{\left|t\right|^{\frac{1}{2\alpha}}}\right)^{2\alpha Q^{*}-2\alpha}\\ &= \frac{g(t)\left|g(t)\right|^{2\alpha-2}}{\sqrt{\left(1+2\alpha \left|g(t)\right|^{2(2\alpha -1)}\right)}}\left(\frac{\left|g(t)\right|}{\left|t\right|^{\frac{1}{2\alpha}}}\right)^{2\alpha Q^{*}-2\alpha} \to (2\alpha)^{\frac{2}{Q-2}} \text{ as } t \to \infty
		\end{align*}
		and $ g^{\prime}(t) g(t) \left|g(t)\right|^{q-2}\left|t\right|^{1-Q^{*}} \to 0$ as $t \to \infty.$ Moreover, $V(\xi)\left|t\right|^{2-Q^{*}} \to 0$ and $V(\xi)g(t)g^\prime(t)\left|t\right|^{1-Q^{*}}\to0$ uniformly in $\xi$, by $(g_2)$, $(g_4)$ and $V\leq V_M$. The first limit cancels against $-(2\alpha)^{\frac{2}{Q-2}}\left|t\right|^{Q^{*}-2}t\left|t\right|^{1-Q^{*}}$, which settles $(H_4)$. \qed
	\end{proof}

	\begin{rem}{\label{rem:H}}
		Fix $\epsilon>0$. By Lemma \ref{lemma:lH} and the continuity of $H$ and $h$, there exist $C_\epsilon, D_\epsilon>0$, independent of $\xi$, such that
		\begin{align*}
			\left|H(\xi,v)\right| \leq \epsilon(v^2 +\left|v\right|^{Q^{*}}) + C_\epsilon \left|v\right|^{\frac{q}{2\alpha}}
		\end{align*}
		and
		\begin{align*}
			\left|h(\xi,v)v\right| \leq \epsilon(v^2 +\left|v\right|^{Q^{*}}) + D_\epsilon \left|v\right|^{\frac{q}{2\alpha}}, \quad \text{ where } 1< \frac{q}{2\alpha} <Q^{*}.
		\end{align*}
	\end{rem}

	\begin{lemma}{\label{lemma:l2}}
		Let  $S_R= \{v \in S^2_1(\H): \|v\|=R \}.$ Then there exist $R, \ba >0$ such that
		\begin{align*}
			\I(v) \geq \ba \ \text{ for all } v \in S_R, \quad \I(v)>0 \ \text{ for all } 0<\|v\| \leq R.
		\end{align*}
	\end{lemma}
	\begin{proof}
		Write $k= \min\{1, V_0\}$ and let $C_m$ denote the norm of the embedding $S^2_1(\H)\hookrightarrow L^m(\H)$, $m \in [2,Q^{*}]$. By \eqref{eq:I3} and Remark \ref{rem:H},
		\begin{align*}
			\I(v) \geq \frac{k}{2}\|v\|^2 - \frac{(2\alpha)^{\frac{2}{Q-2}}}{Q^{*}}\int_\H \left|v\right|^{Q^{*}}d\xi - \epsilon \int_\H \left(v^2+\left|v\right|^{Q^{*}}\right)d\xi - C_\epsilon \int_\H \left|v\right|^{\frac{q}{2\alpha}}d\xi.
		\end{align*}
		Fix $\epsilon>0$ so small that $\epsilon(1+C_{Q^{*}}^{Q^{*}}) \leq \frac{k}{4}$. Since $\frac{q}{2\alpha}$ and $Q^{*}$ both lie in $(2,Q^{*}]$, the embedding gives
		\begin{align*}
			\I(v) \geq \frac{k}{4}\|v\|^2 - C_1 \|v\|^{Q^{*}}- C_2 \|v\|^{\frac{q}{2\alpha}}.
		\end{align*}
		Since $Q^{*}>2$ and  $q>4\alpha$, the function $s \mapsto \frac{k}{4}s^2-C_1s^{Q^{*}}-C_2s^{\frac{q}{2\alpha}}$ is positive on $(0,R]$ for small $R>0$, and setting $\ba$ equal to its value at $s=R$ gives both assertions. \qed
	\end{proof}

	\begin{lemma}{\label{lemma:l3}}
		There exists $v \in S^2_1(\H)$ such that $\|v\|>R$ and $\I(v)<0.$
	\end{lemma}
	\begin{proof}
		It is sufficient to show that there exists $\psi \in S^2_1(\H)$ such that $\I(t \psi) \to - \infty$ as $t \to \infty.$ Choose $\psi \in C_0^\infty(\H)$ such that $0 \leq \psi \leq 1$ and $\operatorname{supp}(\psi)= \overline{B_1(0)}.$ By $(g_3)$ the map $s \mapsto \frac{g(s)}{s}$ is nonincreasing on $(0,\infty)$, so for $\xi \in B_1(0)$ and $s>0$ we get
		\begin{align*}
			\frac{g(s \psi(\xi))}{s \psi(\xi)} \geq \frac{g(s)}{s}, \quad \text{ hence } \left|g(s\psi(\xi))\right| \geq \left|g(s)\right|\psi(\xi).
		\end{align*}
		This gives us
		\begin{align*}
			\I(t \psi)=& \frac{1}{2} \int_\H \left|\nah (t \psi )\right|^2d\xi + \frac{1}{2} \int_\H V(\xi)g^2(t \psi)d\xi- \int_\H F(g(t \psi))d\xi \\ =& \frac{1}{2} \int_\H \left|\nah (t \psi )\right|^2d\xi + \frac{1}{2} \int_\H V(\xi)g^2(t \psi)d\xi - \frac{\lambda}{q} \int_\H \left|g(t \psi)\right|^{q} d\xi - \frac{1}{p} \int_\H \left|g(t \psi)\right|^{p} d\xi \\ \leq & \frac{t^2}{2} \left(\int_{B_1(0)} \left(\left|\nah \psi \right|^2+ V(\xi) \psi^2\right)d\xi -C_1 \frac{\left|g(t)\right|^{q} }{t^2} \int_{B_1(0)} \psi^{q} d\xi -C_2 \frac{\left|g(t)\right|^{p}}{t^2}\int_{B_1(0)} \psi^{p} d\xi \right).
		\end{align*}
		Using $(g_5)$, we observe that  $\lim\limits_{t \to \infty} \frac{\left|g(t)\right|^r}{t^2}= + \infty$ for $r>4\alpha$. Both $q>4\alpha$ and $p=2\alpha Q^{*}>4\alpha$ satisfy this. Hence $\I(t \psi) \to -\infty$ as $t \to \infty.$  \QED
	\end{proof}

	The next two lemmas describe the fibering map $t \mapsto \I(tv)$. They are used in Section 5 to locate the maximum of $\I$ along the ray through the test function.

	\begin{lemma}{\label{lemma:gg}}
		Let $\alpha>\frac12$. Then the map $s \mapsto \dfrac{g(s)g^{\prime}(s)}{s}$ is strictly decreasing on $(0,\infty)$.
	\end{lemma}
	\begin{proof} 
		Let $\mathcal{G}(s)=\dfrac{g(s)g^{\prime}(s)}{s}$ and $G(s)=g^2(s)$ for $s>0$. Then, \[\mathcal{G}^\prime(s)=\frac{s\left(g(s)g^{\prime\prime}(s)+g^\prime(s)^2\right)-g(s)g^\prime(s)}{s^2}= \frac{sG^{\prime \prime}(s)-G^\prime(s)}{2s^2}.\] Thus our assertion is $sG^{\prime\prime}(s)<G^{\prime}(s)$ for $s>0$. Differentiating $(g^{\prime})^{-2}=1+2\alpha g^{4\alpha-2}$ gives $g^{\prime\prime}=-\alpha(4\alpha-2)\,g^{4\alpha-3}(g^{\prime})^{4}$, so, writing $X=2\alpha g^{4\alpha-2}>0$ and using $(g^{\prime})^{2}=(1+X)^{-1}$, we get
		\begin{align*}
			gg^{\prime\prime}=-(2\alpha-1)X(g^{\prime})^{4}, \quad
			(g^{\prime})^{2}+gg^{\prime\prime}=(g^{\prime})^{2}\,\frac{1+(2-2\alpha)X}{1+X}.
		\end{align*}
		Since $2-2\alpha<1$, the last fraction is $<1$. Therefore, using $(g_3)$ in the form $sg^{\prime}\leq g$,
		\begin{align*}
			sG^{\prime\prime}(s)=2s\left((g^{\prime})^{2}+gg^{\prime\prime}\right)
			<2s(g^{\prime})^{2}\leq 2gg^{\prime}=G^{\prime}(s). \qed
		\end{align*}
	\end{proof}

	\begin{lemma}{\label{lemma:fibering}}
		For every $v \in S^2_1(\H)\setminus\{0\}$, there is a unique $t(v)>0$ with $\I^{\prime}(t(v)v)v=0$, and $\I(t(v)v)=\max_{t>0}\I(tv)$.
	\end{lemma}
	\begin{proof}
		For $t>0$, set $\varphi_v(t)=\frac{1}{t}\I^{\prime}(tv)v$. Since $g$ is odd, $g(tv)g^{\prime}(tv)v=\left|v\right|g(t\left|v\right|)g^{\prime}(t\left|v\right|)$, and writing $\Phi_{\ba}(s)=\frac{g(s)^{\ba}g^{\prime}(s)}{s}$ for $s>0$,
		\begin{align*}
			\varphi_v(t)=\left\|\nah v\right\|_2^{2}
			+\int_\H v^2\Big[V(\xi)\Phi_1(t\left|v\right|)
			-\lambda\Phi_{q-1}(t\left|v\right|)-\Phi_{2\alpha Q^{*}-1}(t\left|v\right|)\Big]d\xi.
		\end{align*}
		By Lemma \ref{lemma:gg}, $\Phi_1$ is strictly decreasing. Since $q-1>4\alpha-1$ and $2\alpha Q^{*}-1>4\alpha-1$, by $(g_8)$, $\Phi_{q-1}$ and $\Phi_{2\alpha Q^{*}-1}$ are strictly increasing. Hence $t \mapsto \varphi_v(t)$ is strictly decreasing.

		By $(g_3)$ and $(g_4)$, $0<\Phi_1\leq 1$, and $\Phi_1(s)\to1$, $\Phi_{q-1}(s), \Phi_{2\alpha Q^{*}-1}(s)\to 0$ as $s \to 0^{+}$. Then, dominated convergence gives $\varphi_v(0^{+})=\left\|\nah v\right\|_2^{2}+\int_\H V v^2 d\xi>0$. By $(g_5)$, $\Phi_{2\alpha Q^{*}-1}(s)\sim (2\alpha)^{\frac{2}{Q-2}}s^{Q^{*}-2}\to\infty$, as $s \to \infty$, so Fatou's lemma gives $\varphi_v(t)\to-\infty$. Thus $\varphi_v$ has a unique zero $t(v)$, with $\varphi_v>0$ on $(0,t(v))$ and $\varphi_v<0$ on $(t(v),\infty)$. Since $\frac{d}{dt}\I(tv)=t\varphi_v(t)$, the map $t \mapsto \I(tv)$ increases on $(0,t(v))$ and decreases afterwards. \qed
	\end{proof}

	In light of Lemma~\ref{lemma:l2} and Lemma~\ref{lemma:l3} together with a variant of the Ambrosetti--Rabinowitz mountain pass theorem \cite{schechter2012linking}, we have the following consequence. Let
	\begin{align*}
		\Ga=\{\sigma \in C([0,1], S^2_1(\H)): \sigma (0)=0, \sigma (1) \neq 0, \I(\sigma (1))<0\}
	\end{align*}
	and
	\begin{equation}{\label{eq:C}}
		c= \inf_{\sigma \in \Ga} \sup_{t\in [0,1]} \I(\sigma (t)).
	\end{equation}
	By Lemma~\ref{lemma:l2}, $\I>0$ on $\{0<\|v\|\leq R\}$, so $\I(\sigma(1))<0$ forces $\|\sigma(1)\|>R$ for every $\sigma \in \Ga$; since $\|\sigma(0)\|=0$, the map $t \mapsto \|\sigma(t)\|$ is continuous and therefore $\sigma([0,1])$ meets $S_R$. Consequently
	\begin{equation}{\label{eq:cpos}}
		c \geq \ba>0 .
	\end{equation}
	Then, for the constant $c$, there exists a  sequence $\{v_n\} \subset S^2_1(\H)$ at level $c$ satisfying
	\begin{align*}
		\I(v_n) \to c, \quad  (1+\|v_n\|) \|\I^{\prime}(v_n)\| \to 0 \quad \text{ as } n \to \infty.
	\end{align*}
	This sequence $\{v_n\}$ is called a Cerami sequence.

	\begin{lemma}{\label{lemma:l4}}
		The Cerami sequence $\{v_n\}$ is bounded in $S^2_1(\H).$
	\end{lemma}
	\begin{proof}
		By definition of $\{v_n\}$,
		\begin{small}
			\begin{align}
				I_{\lambda}(v_n)=&  \frac{1}{2} \int_\H \left|\nah v_n\right|^2d\xi + \frac{1}{2} \int_\H V(\xi)g^2(v_n)d\xi- \int_\H F(g(v_n))d\xi= c+ e_n, \label{eq:E1} \\ \left|(1+\|v_n\|) \I^{\prime}(v_n)z \right| &= \left|(1+\|v_n\|) \left(\int_\H \nah v_n \cdot \nah z\, d\xi + \int_\H [V(\xi)g(v_n)z-f(g(v_n))z]g^{\prime}(v_n) d\xi \right)\right| \notag \\& \leq \epsilon_n \|z\|, \label{eq:E2}
			\end{align}
		\end{small}
		for all $z\in S^2_1(\H)$, where $e_n , \epsilon _n \to 0$ as $n \to \infty.$ Let $z=z_n= \frac{g(v_n)}{g^{\prime}(v_n)}.$ Solving the  differential equation for $g$ gives
		\begin{align*}
			\left(\frac{g}{g^\prime}\right)^\prime(s)=1+\frac{2\alpha (2\alpha -1) \left|g(s)\right|^{2(2\alpha -1)}}{1+2\alpha \left|g(s)\right|^{2(2\alpha -1)}} \in [1,2\alpha],
		\end{align*}
		so that
		\begin{small}
			\begin{align*}
				\quad &(1+\|v_n\|)\left(\int_\H \left(1+\frac{2\alpha (2\alpha -1) \left|g(v_n)\right|^{2(2\alpha -1)}}{(1+2\alpha \left|g(v_n)\right|^{2(2\alpha -1)})}\right) \left|\nah v_n \right|^2 d\xi  + \int_\H [V(\xi) g^2(v_n)- f(g(v_n))g(v_n)]d\xi \right)\\&\leq \epsilon_n \|z_n\|.
			\end{align*}
		\end{small}
		By $(g_4)$ and the bound $\left|g\right|\sqrt{1+2\alpha\left|g\right|^{2(2\alpha-1)}} \leq \left|g\right| + \sqrt{2\alpha}\left|g\right|^{2\alpha} \leq (1+2\alpha)\left|s\right|$, we have \[ \left|z_n\right|_2 \le C \left|v_n\right|_2, \left|\nah z_n \right| \le 2\alpha \left|\nah v_n\right| \text{ and } \|z_n\| \leq C \|v_n\|.\] Dividing by $1+\|v_n\|$ and using $\|z_n\|\leq C\|v_n\| \leq C(1+\|v_n\|)$, it implies
		\begin{small}
			\begin{align}
				\left| \I^{\prime}(v_n)z_n \right| &= \left| \left(\int_\H \left(1+\frac{2\alpha (2\alpha -1) \left|g(v_n)\right|^{2(2\alpha -1)}}{(1+2\alpha \left|g(v_n)\right|^{2(2\alpha -1)})}\right) \left|\nah v_n \right|^2 d\xi + \int_\H [V(\xi) g^2(v_n)- f(g(v_n))g(v_n)]d\xi \right)\right| \nonumber \\& \leq  \epsilon_n.\label{eq:E3}
			\end{align}
		\end{small}
		Since $q< 2\alpha Q^{*}$, one can easily follow that
		\begin{equation}{\label{eq:E4}}
			q F(t)-f(t)t= \left(\frac{q}{2\alpha Q^{*}}-1\right)\left|t\right|^{2\alpha Q^{*}} \leq 0 \text{ for all } t\in \R, \text{ with equality only at } t=0.
		\end{equation}
		Taking into account \eqref{eq:E1}, \eqref{eq:E2}, \eqref{eq:E3} and \eqref{eq:E4}, we deduce that
		\begin{small}
			\begin{align*}
				\int_\H & \left\{ \frac{1}{2}- \frac{1}{q}\left(1+\frac{2\alpha (2\alpha -1) \left|g(v_n)\right|^{2(2\alpha -1)}}{(1+2\alpha \left|g(v_n)\right|^{2(2\alpha -1)})}\right) \right\} \left|\nah v_n \right|^2  d\xi + \left(\frac{1}{2}-\frac{1}{q}\right) \int_\H V(\xi)g^2(v_n) d\xi
				\\ =& \left(\frac{1}{2} \int_\H \left(\left|\nah v_n\right|^2 +V(\xi)g^2(v_n)\right)d\xi- \int_\H F(g(v_n))d\xi \right) + \int_\H \left(F(g(v_n))- \frac{1}{q}f(g(v_n))g(v_n)\right) d\xi\\& - \frac{1}{q} \left\{\int_\H \left(1+\frac{2\alpha (2\alpha -1) \left|g(v_n)\right|^{2(2\alpha -1)}}{(1+2\alpha \left|g(v_n)\right|^{2(2\alpha -1)})}\right)  \left|\nah v_n \right|^2  d\xi+ \int_\H \left(V(\xi)g^2(v_n)-  f(g(v_n))g(v_n)\right)d\xi \right\} \\ \le &  c+e_n+ \epsilon _n.
			\end{align*}
		\end{small}
		The bracket in the first integral is at least $\frac{1}{2}-\frac{2\alpha}{q}$, which is positive because $q>4\alpha$. Hence
		\begin{align*}
			\int_\H \left(\left|\nah v_n\right|^2+ V(\xi)g^2(v_n)\right)d\xi \leq C \text{ for some } C>0,
		\end{align*}
		and, by \eqref{eq:E1}, $\int_\H F(g(v_n))d\xi \leq C$ as well.
		Define $A_n=\{\xi\in \H: \left|v_n(\xi)\right| \leq 1 \}$ and $B_n= \{\xi \in \H: \left|v_n(\xi)\right|>1\}$. Then
		\begin{align*}
			\int_\H \left|v_n\right|^2d\xi= \int_{A_n} \left|v_n\right|^2d\xi + \int_{B_n} \left|v_n\right|^2d\xi.
		\end{align*}
		By definition of $F$, $F(t) \geq Ct^{q}$ for some $C>0$ and every $t \geq 1$. Using $(g_6)$ and $q>4\alpha$, we obtain $F(g(t)) \geq C \left|g(t)\right|^{q} \geq Ct^{\frac{q}{2 \alpha}} \geq Ct^2$ for every $t \geq 1$, and $\left|g(t)\right| \geq C\left|t\right|$ for some $C>0$ and every $\left|t\right| \leq 1.$ Thus
		\begin{align*}
			\int_{B_n} \left|v_n\right|^2d\xi \leq \frac{1}{C} \int_{B_n} F(g(v_n))d\xi \leq \frac{1}{C} \int_\H F(g(v_n))d\xi
		\end{align*}
		and
		\begin{align*}
			\int_{A_n} \left|v_n\right|^2d\xi \leq \frac{1}{C} \int_{A_n} g^2(v_n)d\xi \leq \frac{1}{CV_0} \int_\H V(\xi)g^2(v_n)d\xi,
		\end{align*}
		which shows that $\{v_n\}$ is bounded in $L^2(\H)$. Combined with the bound on $\int_\H \left|\nah v_n\right|^2d\xi$ obtained above, this gives the boundedness of $\{v_n\}$ in $S^2_1(\H)$. \QED
	\end{proof}

	\begin{lemma}{\label{lemma:l6}}
		Let $a$ and $b$ be such that $1<a<Q$, $1 \leq b<\infty$ and $b\neq \frac{aQ}{Q-a}.$ Let $\{u_n\}$ be a sequence with the property that $\{u_n\}$ is bounded in $L^{b}(\H)$, $\{\nah u_n\}$ is bounded in $L^{a}(\H)$ and
		\begin{align*}
			\sup_{\xi \in \H} \int_{B_R(\xi)} \left|u_n\right|^{b} d\xi \to 0 \quad \text{ as } n \to \infty,
		\end{align*}
		for some $R>0$, where $B_R(\xi)$ is the gauge ball centred at $\xi$. Then $u_n \to 0$ in $L^d(\H)$, for $d \in (b, \frac{aQ}{Q-a})$.
	\end{lemma}
	\begin{proof}
		The proof of \cite[Lemma~I.1]{lions1984concentration} is a covering argument combined with the Sobolev and H\"older inequalities on balls of a fixed radius. It carries over directly  once one uses the Folland--Stein inequality in place of the Sobolev inequality, and observes that the gauge balls $\{B_R(\xi)\}_{\xi \in \H}$ satisfy the doubling property $\left|B_{2R}(\xi)\right|=2^{Q}\left|B_R(\xi)\right|$ and admit, for each fixed $R$, a covering of $\H$ with uniformly bounded overlap. We omit the details. \qed
	\end{proof}

	\section{Minimax level}
	In this section, we establish the upper bound on the minimax level.
    Set
    \begin{equation}{\label{eq:qstar}}
q_* := 2\alpha Q^{*}-2\alpha\,\min\left\{1,\ 2-\tfrac{1}{\alpha}\right\}
=\begin{cases} 2\alpha Q^{*}+2-4\alpha, & \tfrac12<\alpha<1,\\[2pt]
2\alpha Q^{*}-2\alpha, & \alpha \geq 1,\end{cases}
\end{equation}
	and assume throughout this section that
	\begin{equation}{\label{eq:qcond}}
		\max\{4\alpha,\,q_*\}<q<2\alpha Q^{*}.
	\end{equation}

	\begin{rem}{\label{rem:qstar}}
    For $\alpha \geq 1$, the restriction \eqref{eq:qcond} is therefore effective only when
$N=1$, while for $\tfrac12<\alpha<1$, it is effective precisely when
$Q<\frac{2(4\alpha-1)}{2\alpha-1}$, a threshold which increases without bound as
$\alpha \downarrow \tfrac12$. For instance, when $N=1$ and $\alpha=1$, \eqref{eq:qcond}
reads $q>6$ within  $(4,8)$.
	\end{rem}

	\begin{lemma}{\label{lemma:l5}}
		Assume \eqref{eq:qcond}. The minimax level $c$ defined in \eqref{eq:C} satisfies
		\begin{align*}
			c< \frac{S^{Q/2}}{2\alpha Q}.
		\end{align*}
	\end{lemma}
	\begin{proof}
		It is enough to prove that
		\begin{align*}
			\sup_{t \geq 0} \I(tw) < \frac{S^{Q/2}}{2\alpha Q}, \text{ for some } w \in S^2_1(\H)\setminus\{0\}.
		\end{align*}
		Fix $R_0>0$ and choose $\phi \in C_c^\infty(\H)$ such that $0\leq \phi \leq 1$, $\phi\equiv1$ on $B_{R_0}(0)$ and $\phi\equiv0$ on $\H \setminus B_{2R_0}(0).$ Set $R_\mu=\mu^{\beta}$, where
		\begin{align*}
			\frac{1}{2}<\beta< \frac{q}{2\alpha Q^{*}},
		\end{align*}
		so that $R_\mu<R_0$ for all small $\mu$. Let
		\begin{align*}
			W_\mu(\xi)= \frac{C \mu^{\frac{Q-2}{2}}}{\left(t^2+(\mu^2+\left|x\right|^2+\left|y\right|^2)^2\right)^{\frac{Q-2}{4}}}, \quad \xi =(x,y,t)\in \H.
		\end{align*}
		It follows from \cite{citti1995semilinear} that, for a suitable choice of $C=C(Q)$, the function $W_\mu$ is an entire solution of $-\Delta_{\mathbb{H}} u=u^{Q^{*}-1}$ and satisfies
		\begin{align*}
			\int_\H \left|\nah W_\mu \right|^2 d\xi =\int_\H \left| W_\mu \right|^{Q^{*}}d\xi= S^{\frac{Q}{2}},
		\end{align*}
		and
		\begin{align*}
			\int_{\H \setminus B_{R_0}(0)} \left|\nah W_\mu \right|^2 d\xi = O\left({\mu}^{Q-2}\right), \quad \int_{\H \setminus B_{R_0}(0)} \left|W_\mu \right|^{Q^{*}} d\xi = O\left(\mu^{Q}\right) \quad \text{ as } \mu \to0.
		\end{align*}

		Define
		\begin{align*}
			v_\mu(\xi)= \frac{\phi(\xi)W_\mu(\xi)}{\left(\int_{B_{2R_0}(0)} \left|\phi W_\mu \right|^{Q^{*}} d\xi\right)^{\frac{1}{Q^{*}}}} \quad \text{ and } \quad T_\mu = \int_\H \left|\nah v_\mu \right|^2 d\xi.
		\end{align*}
		Then $\int_\H \left|v_\mu\right|^{Q^{*}}d\xi=1$, and using the above estimates we conclude that
		\begin{equation}{\label{eq:E5}}
			T_\mu= S +O\left({\mu}^{Q-2}\right).
		\end{equation}

		\textbf{Claim 1:} There exist  $d_1, d_2>0$, independent of $\mu$, and $\mu_0>0$ such that
		\begin{align*}
			d_1 \leq \frac{g(v_\mu)}{{v_\mu}^\frac{1}{2\alpha}} \leq d_2  \text{ for all } \mu \in (0,\mu_0) \text{ and } \xi \in B_{R_\mu}(0).
		\end{align*}
		By Lemma~\ref{lemma:l1}, $\lim_{s\to \infty} g(s)s^{-\frac{1}{2\alpha}}=(2\alpha)^{\frac{1}{4\alpha}}$, so there exists $t_0 >0$ such that
		\begin{equation}{\label{eq:E6}}
			(2\alpha)^{\frac{1}{4\alpha}}-\ga < \frac{g(t)}{t^\frac{1}{2\alpha}}< (2\alpha)^{\frac{1}{4\alpha}}+\ga  \text{ for all } t \geq t_0,
		\end{equation}
		where $0<\ga < (2\alpha)^{\frac{1}{4\alpha}}.$ For all $\xi \in B_{R_\mu}(0)$ and $\mu$ small, we have $\phi(\xi)=1$ and $\varrho(\xi)<R_\mu$, so
		\begin{align*}
			v_\mu(\xi)&= \frac{W_\mu(\xi)}{\left(\int_{B_{2R_0}(0)} \left|\phi W_\mu \right|^{Q^{*}} d\xi\right)^{\frac{1}{Q^{*}}}} \\& \geq \frac{W_\mu(\xi)}{\left(\int_\H \left|W_\mu \right|^{Q^{*}} d\xi\right)^{\frac{1}{Q^{*}}}}=\frac{W_\mu(\xi)}{\left(S^{\frac{Q}{2}}\right)^{\frac{1}{Q^{*}}}}= \frac{1}{S^{\frac{Q-2}{4}}}\frac{C \mu^{\frac{Q-2}{2}}}{\left(t^2+(\mu^2+\left|x\right|^2+\left|y\right|^2)^2\right)^{\frac{Q-2}{4}}} \\& \geq \frac{1}{S^{\frac{Q-2}{4}}} \frac{C \mu^{\frac{Q-2}{2}}}{\left(\mu^{4\beta}+(\mu^2+\mu^{2\beta})^2\right)^{\frac{Q-2}{4}}}\\& \geq \frac{1}{S^{\frac{Q-2}{4}}} \frac{C }{\mu^{(2\beta-1)\frac{(Q-2)}{2}}\left(1+(\mu^{2-2\beta}+1)^2\right)^{\frac{Q-2}{4}}} \to \infty \text{ as } \mu \to 0.
		\end{align*}
		Therefore there exists $\mu_0>0$, independent of $\xi$, such that
		\begin{equation}{\label{eq:E7}}
			v_\mu(\xi) \geq t_0, \text{ for all } \mu \in (0, \mu_0) \text{ and } \xi \in B_{R_\mu}(0).
		\end{equation}
		Combining \eqref{eq:E6} and \eqref{eq:E7} yields the claim.

		By \eqref{eq:I3},
		\begin{align*}
			\I(tv_\mu)= \frac{t^2}{2} \int_\H \left(\left|\nah v_\mu\right|^2 + V(\xi)v_\mu^2\right)d\xi- t^{Q^{*}}\frac{(2\alpha)^{\frac{2}{Q-2}}}{ Q^{*}} \int_\H \left|v_\mu\right|^{Q^{*}} d\xi- \int_\H H(\xi,tv_\mu)d\xi.
		\end{align*}
		By Lemma \ref{lemma:fibering}, the maximum of $t \mapsto \I(tv_\mu)$ over $t>0$ is attained at the unique $t_\mu:=t(v_\mu)$ on  solving $\I^{\prime}(t_\mu v_\mu)v_\mu=0$, that is,
		\begin{equation}{\label{eq:tmueq}}
			t_\mu^2\left(T_\mu + \int_\H V(\xi)v_\mu^2 d\xi\right)= (2\alpha)^{\frac{2}{Q-2}} t_\mu^{Q^{*}} + \int_\H h(\xi,t_\mu v_\mu)\,t_\mu v_\mu\, d\xi.
		\end{equation}

		\textbf{Claim 2:} there are $0<c_1<c_2<\infty$, independent of $\mu$, with $c_1 \leq t_\mu \leq c_2$ for all small $\mu$.

		Write $a_\mu=T_\mu+\int_\H Vv_\mu^2d\xi$, $b_\mu=\int_\H v_\mu^2 d\xi$ and $e_\mu=\int_\H \left|v_\mu\right|^{q/2\alpha} d\xi$. By Remark \ref{rem:H}, for any $\epsilon>0$,
		\begin{align*}
			\left|\int_\H h(\xi,t_\mu v_\mu)t_\mu v_\mu d\xi\right| \leq \epsilon t_\mu^{2}b_\mu+\epsilon t_\mu^{Q^{*}}+D_\epsilon t_\mu^{\frac{q}{2\alpha}}e_\mu .
		\end{align*}
		Fix $\epsilon<\frac12 (2\alpha)^{\frac{2}{Q-2}}$. From \eqref{eq:tmueq},
		\begin{align*}
			\tfrac12 (2\alpha)^{\frac{2}{Q-2}}t_\mu^{Q^{*}} \leq (a_\mu+\epsilon b_\mu)t_\mu^{2}+D_\epsilon e_\mu t_\mu^{\frac{q}{2\alpha}},
		\end{align*}
		and since $2<\frac{q}{2\alpha}<Q^{*}$ while $a_\mu,b_\mu,e_\mu$ are bounded, $t_\mu \leq c_2$. Dividing \eqref{eq:tmueq} by $t_\mu^{2}$ and using the same bound,
		\begin{align*}
			a_\mu-\epsilon b_\mu \leq \left((2\alpha)^{\frac{2}{Q-2}}+\epsilon\right)t_\mu^{Q^{*}-2}+D_\epsilon e_\mu t_\mu^{\frac{q}{2\alpha}-2}.
		\end{align*}
		By \eqref{eq:E5}, $a_\mu \to S>0$, while $b_\mu=O(\mu^{2})$ and $e_\mu=O(\mu^{Q-\frac{q(Q-2)}{4\alpha}})\to0$ because $q<2\alpha Q^{*}$. Hence the left-hand side is at least $\frac{S}{2}$ for small $\mu$, which forces $t_\mu \geq c_1>0$ and proves Claim 2.

Since $\int_\H \left|v_\mu\right|^{Q^{*}}d\xi=1$, \eqref{eq:I3} gives, for every $t>0$,
\begin{align*}
\I(tv_\mu)=\frac{t^{2}}{2}\,a_\mu-\frac{(2\alpha)^{\frac{2}{Q-2}}}{Q^{*}}\,t^{Q^{*}}
-\int_\H H(\xi,tv_\mu)\,d\xi .
\end{align*}
The elementary maximisation
$\sup_{t>0}\big(\frac{t^{2}}{2}A-\frac{B}{Q^{*}}t^{Q^{*}}\big)
=\frac{A^{Q/2}}{Q\,B^{\frac{Q-2}{2}}}$, applied with $A=a_\mu$ and
$B=(2\alpha)^{\frac{2}{Q-2}}$, for which $B^{\frac{Q-2}{2}}=2\alpha$, together with
Claim 2, yields
\begin{align}
\I(t_\mu v_\mu) \leq& \frac{a_\mu^{Q/2}}{2\alpha Q}
-\inf_{t\in[c_1,c_2]}\int_\H H(\xi,tv_\mu)\,d\xi \notag \\
=& \frac{1}{2\alpha Q}\left(S+O\left(\mu^{Q-2}\right)
+\int_\H V(\xi)v_\mu^2\,d\xi\right)^{\frac{Q}{2}}
-\inf_{t\in[c_1,c_2]}\int_\H H(\xi,tv_\mu)\,d\xi \notag \\
\leq& \frac{S^{\frac{Q}{2}}}{2\alpha Q}+O\left(\mu^{Q-2}\right)
+K_1\int_\H V(\xi)v_\mu^2\,d\xi
-\inf_{t\in[c_1,c_2]}\int_\H H(\xi,tv_\mu)\,d\xi , \label{eq:E8}
\end{align}
where the last step uses \eqref{eq:E5} and the elementary inequality
$(x+y)^{m}\leq x^{m}+my(x+y)^{m-1}$ for $x,y>0$, $m\geq1$, with $x=S$, $m=\frac{Q}{2}$
and $y=O(\mu^{Q-2})+\int_\H V(\xi)v_\mu^2\,d\xi$, which is bounded; here
$K_1:=\frac{1}{4\alpha}\big(S+\sup_\mu y_\mu\big)^{\frac{Q-2}{2}}$ depends only on
$Q,\alpha,S$ and $V_M$.\\

		\textbf{Claim 3:} $\lim\limits_{\mu \to 0^+}\frac{1}{\mu^{r}}\int_{B_{2R_0}(0)} \left(K_1 V(\xi) v_\mu^2- H(\xi,tv_\mu)\right)d\xi=-\infty$ uniformly for $t \in [c_1,c_2]$, where $r=Q-2.$

		Since $t$ ranges in a fixed compact subset of $(0,\infty)$, it only rescales the constants below; we suppress it and write $v_\mu$ for $tv_\mu$. Let
		\begin{align*}
			I_1= \frac{1}{\mu^r}\int_{B_{R_\mu}(0)}\left(K_1 V(\xi) v_\mu^2- H(\xi,v_\mu)\right)d\xi
		\end{align*}
		and
		\begin{align*}
			I_2= \frac{1}{\mu^r}\int_{{B_{2R_0}(0)}\setminus{B_{R_\mu}(0)}}\left(K_1 V(\xi) v_\mu^2- H(\xi,v_\mu)\right)d\xi.
		\end{align*}
		Using Claim $1$ and $V(\xi) \leq V_M$,
		\begin{small}
			\begin{align}
				I_1=& \frac{1}{\mu^r}\int_{B_{R_\mu}(0)}\left(K_1 V(\xi) v_\mu^2- H(\xi,v_\mu)\right)d\xi \notag \\=& \frac{1}{\mu^r}\int_{B_{R_\mu}(0)}\left(K_1 V(\xi) v_\mu^2- \left(\frac{\lambda \left|g(v_\mu)\right|^{q}}{q}  + \frac{\left|g(v_\mu)\right|^{2\alpha Q^{*}}}{2\alpha Q^{*}} + \frac{1}{2}V(\xi)(v_\mu^2- g^2(v_\mu))- \frac{(2\alpha)^{\frac{2}{Q-2}}}{Q^{*}} \left|v_\mu\right|^{Q^{*}}\right)\right)d\xi \notag \\ \leq & \frac{1}{\mu^r}\int_{B_{R_\mu}(0)}\left(K_1 V_M v_\mu^2- \frac{ \lambda}{q } \left(\frac{ \left|g(v_\mu)\right|}{v_\mu^{\frac{1}{2\alpha}}}\right)^{q}v_\mu^{\frac{q}{2\alpha}}+\Theta(v_\mu)\right)d\xi \notag \\ \leq & \frac{1}{\mu^r}\int_{B_{R_\mu}(0)}\left(K_1 V_M v_\mu^2- \frac{\lambda d_1^{q}}{q } v_\mu^{\frac{q}{2\alpha}}+\Theta(v_\mu)\right)d\xi,\notag
			\end{align}
		\end{small}
		where
		\begin{equation}{\label{eq:Theta}}
			\Theta(s):=\frac{(2\alpha)^{\frac{2}{Q-2}}}{Q^{*}}\left|s\right|^{Q^{*}}-\frac{\left|g(s)\right|^{2\alpha Q^{*}}}{2\alpha Q^{*}} \;=\; \frac{1}{2\alpha Q^{*}}\left[\left(\sqrt{2\alpha}\,\left|s\right|\right)^{Q^{*}}-\left|g(s)\right|^{2\alpha Q^{*}}\right] \;\geq\; 0,
		\end{equation}
		the second equality using $\frac{Q^{*}}{2}-1=\frac{2}{Q-2}$ and the sign following from $(g_4)$.


		We first estimate $\Theta$. Put $\psi(s)=\sqrt{2\alpha}\,s-g(s)^{2\alpha}$ for $s>0$, so that $\psi(0)=0$ and, by the equation for $g$,
		\begin{align*}
			\psi^{\prime}(s)=\sqrt{2\alpha}-2\alpha g(s)^{2\alpha-1}g^{\prime}(s)
				=\sqrt{2\alpha}\left[1-\left(1+\tfrac{1}{2\alpha}g(s)^{2-4\alpha}\right)^{-\frac12}\right]>0 .
		\end{align*}
		By $(g_5)$, $g(s)^{2-4\alpha}\sim (2\alpha)^{\frac{1-2\alpha}{2\alpha}}s^{\frac1\alpha-2}$ as $s \to \infty$, and since $\frac1\alpha-2<0$ we get $\psi^{\prime}(s)\asymp s^{\frac1\alpha-2}$, whence
		\begin{equation}{\label{eq:psiasym}}
			\psi(s)\asymp \begin{cases} s^{\frac1\alpha-1}, & \tfrac12<\alpha<1,\\[2pt] \ln s, & \alpha=1,\\[2pt] 1, & \alpha>1,\end{cases} \qquad \text{ as } s \to \infty .
		\end{equation}
		Expanding $\big(1-\tfrac{\psi(s)}{\sqrt{2\alpha}\,s}\big)^{Q^{*}}$ in \eqref{eq:Theta} gives $\Theta(s)\asymp s^{Q^{*}-1}\psi(s)$, so by \eqref{eq:psiasym}
		\begin{equation}{\label{eq:Thetabound}}
			\Theta(s) \leq C\,s^{\sigma}\left(1+\ln^{+}s\right) \quad \text{ for } s \geq 1, \qquad \sigma:=Q^{*}-1+\left(\tfrac1\alpha-1\right)^{+},
		\end{equation}
		the logarithm being needed only when $\alpha=1$, where $\sigma=Q^{*}-1$.

		Since $(Q^{*}-1)(Q-2)=Q+2>Q$, the exponent $\sigma$ satisfies $\sigma(Q-2)>Q$, so 
       the computation leading to \eqref{eq:J1est} below, with $\tfrac{q}{2\alpha}$ replaced
by $\sigma$ and the logarithmic factor controlled by
$\ln^{+}v_\mu \leq \ln\big(C\mu^{-\frac{Q-2}{2}}\big) \leq C\ln\tfrac1\mu$, yields
		\begin{equation}{\label{eq:Thetaint}}
			\frac{1}{\mu^{Q-2}}\int_{B_{R_\mu}(0)}\Theta(v_\mu)\,d\xi \;\leq\; C\,\mu^{\,2-\frac{\sigma(Q-2)}{2}}\left(1+\ln\tfrac1\mu\right).
		\end{equation}
		Comparing exponents with the negative term, whose order is $\mu^{\,2-\frac{q(Q-2)}{4\alpha}}$, we see that \eqref{eq:Thetaint} is of strictly lower order precisely when $\frac{q}{2\alpha}>\sigma$, that is, when $q>2\alpha\sigma=q_*$; the logarithm at $\alpha=1$ is absorbed because the inequality $q>q_*$ is strict. This is hypothesis \eqref{eq:qcond}. Consequently
		\begin{align*}
			I_1 \leq \frac{1}{\mu^{Q-2}}\int_{B_{R_\mu}(0)}\left(K_1 V_M v_\mu^2-\frac{\lambda d_1^{q}}{2q}v_\mu^{\frac{q}{2\alpha}}\right)d\xi \quad \text{ for all small } \mu,
		\end{align*}
		the factor $\frac12$ absorbing the contribution of $\Theta$.

		It remains to estimate the right-hand side. For any $C_1, C_2>0$, consider
		\begin{align} {\label{eq:J1est}}
			J_1& =\frac{1}{\mu^{Q-2}}\int_{B_{R_\mu}(0)}\left(C_1 W_\mu^2- C_2 W_\mu^{\frac{q}{2\alpha}}\right)d\xi
			\\&= \frac{C_1}{\mu^r}\int_{B_{R_\mu}(0)} \frac{ \mu^{Q-2}}{\left(t^2+(\mu^2+\left|x\right|^2+\left|y\right|^2)^2\right)^{\frac{Q-2}{2}}} d\xi  - \frac{C_2}{\mu^r}\int_{B_{R_\mu}(0)} \frac{\mu^{\frac{(Q-2)q}{4\alpha}}}{\left(t^2+(\mu^2+\left|x\right|^2+\left|y\right|^2)^2\right)^{{\frac{(Q-2)q}{8\alpha}}}}d\xi  \notag \\& \leq \frac{C_1}{\mu^r}\int_{B_{R_\mu}(0)} \frac{ \mu^{Q-2}}{\left(t^2+(\left|x\right|^2+\left|y\right|^2)^2 +\mu^4\right)^{\frac{Q-2}{2}}} d\xi  - \frac{C_2}{\mu^r}\int_{B_{\frac{R_\mu}{\mu}}(0)} \frac{\mu^{Q-\frac{(Q-2)q}{4\alpha}}}{\left(t^2+(1+\left|x\right|^2+\left|y\right|^2)^2\right)^{{\frac{(Q-2)q}{8\alpha}}}}d\xi  \notag\\& \leq \frac{C_1}{\mu^r} \int_0^{R_\mu} \frac{\mu^{Q-2}}{(\mu^4 + \rho^4)^{\frac{Q-2}{2}}}\rho^{Q-1}d\rho-\frac{C_2 \mu^{Q-\frac{(Q-2)q}{4\alpha}}}{\mu^r}\left(C_3 + \int_1^{\frac{R_\mu}{\mu}}\frac{\rho^{Q-1}}{\rho^{\frac{(Q-2)q}{2\alpha}}}d\rho\right) \notag \\& = \frac{C_1 \mu^2}{\mu^r} \int_0^{\frac{R_\mu}{\mu}} \frac{z^{Q-1}}{(1+z^4)^{\frac{Q-2}{2}}}dz-\frac{C_2 \mu^{Q-\frac{(Q-2)q}{4\alpha}}}{\mu^r} \left(C_3 + C_4 \mu^{(1-\beta)\left(\frac{(Q-2)q}{2\alpha}-Q\right)} \right)\notag \\&= \frac{1}{\mu^{Q-4}} \left(C_1 \int_0^{\frac{R_\mu}{\mu}} \frac{z^{Q-1}}{(1+z^4)^{\frac{Q-2}{2}}}dz- \frac{C_2}{\mu^{(Q-2)\left(\frac{q}{4\alpha}-1 \right)}} \left(C_3+ C_4 \mu^{(1-\beta)\left(\frac{(Q-2)q}{2\alpha}-Q\right)} \right) \right).\notag
		\end{align}
		For $Q>4$,
		\begin{align*}
			\int_0^{\frac{R_\mu}{\mu}} \frac{z^{Q-1}}{(1+z^4)^{\frac{Q-2}{2}}}dz <\infty, \quad \text{ and } \quad \frac{(Q-2)q}{2\alpha}-Q>2(Q-2)-Q=Q-4>0.
		\end{align*}
		Moreover $2-\frac{q(Q-2)}{4\alpha}<4-Q$ because $q>4\alpha$, so the negative term dominates and $J_1 \to -\infty$ as $\mu \to 0.$ For $Q=4$,
		\begin{align*}
			J_1 &\leq  C_1 \int_0^{\frac{R_\mu}{\mu}} \frac{z^3}{1+z^4}dz- \frac{C_2}{\mu^{2\left(\frac{q}{4\alpha}-1 \right)}} \left(C_3+ C_4 \mu^{4(1-\beta)\left(\frac{q}{4\alpha}-1\right)} \right) \\&= C_1 \ln\left(1+\left(\frac{R_\mu}{\mu}\right)^4\right)-\frac{C_2}{\mu^{2\left(\frac{q}{4\alpha}-1 \right)}} \left(C_3+ C_4 \mu^{4(1-\beta)\left(\frac{q}{4\alpha}-1\right)} \right) \\& \to -\infty \text{ as } \mu \to 0, \text{ since } \lim_{\mu \to 0} \frac{\mu^{-2\left(\frac{q}{4\alpha}-1 \right)}}{\ln\left(1+\left(\frac{R_\mu}{\mu}\right)^4\right)}= +\infty.
		\end{align*}
		We now bound $I_2$ from above. By Remark~\ref{rem:H} and $V \leq V_M$,
		\begin{equation}{\label{eq:I2bound}}
			I_2 \leq \frac{C}{\mu^{Q-2}}\int_{{B_{2R_0}(0)}\setminus{B_{R_\mu}(0)}}\left(v_\mu^2+ \left|v_\mu\right|^{\frac{q}{2\alpha}}\right)d\xi + \frac{1}{\mu^{Q-2}}\int_{{B_{2R_0}(0)}\setminus{B_{R_\mu}(0)}}\Theta(v_\mu)d\xi .
		\end{equation}

       Since $W_\mu(\xi)\leq C\mu^{\frac{Q-2}{2}}\varrho(\xi)^{-(Q-2)}$ on the annulus, the
two terms of the first integral must be estimated separately:
\begin{align*}
\int_{B_{2R_0}(0)\setminus B_{R_\mu}(0)}v_\mu^{2}\,d\xi
=\begin{cases}
O\big(\mu^{Q-2}R_\mu^{4-Q}\big)=O\big(\mu^{Q-2-\beta(Q-4)}\big), & Q>4,\\[2pt]
O\big(\mu^{2}\ln\tfrac1\mu\big), & Q=4,
\end{cases}
\end{align*}
while $q>4\alpha \geq \alpha Q^{*}$ for $Q\geq4$ makes the radial exponent
$Q-1-\tfrac{q(Q-2)}{2\alpha}$ strictly less than $-1$, so
\begin{align*}
\int_{B_{2R_0}(0)\setminus B_{R_\mu}(0)}\left|v_\mu\right|^{\frac{q}{2\alpha}}d\xi
= O\left(\mu^{\frac{q(Q-2)}{4\alpha}}R_\mu^{\,Q-\frac{q(Q-2)}{2\alpha}}\right)
= O\left(\mu^{\frac{q(Q-2)}{4\alpha}+\beta\left(Q-\frac{q(Q-2)}{2\alpha}\right)}\right).
\end{align*}
After division by $\mu^{Q-2}$, the first is dominated by the negative term of $J_1$, of
order $\mu^{2-\frac{q(Q-2)}{4\alpha}}$, because $2-\tfrac{q(Q-2)}{4\alpha}<-\beta(Q-4)$
follows from $q>4\alpha$ and $\beta<1$; for the second, the comparison of exponents
reduces to $(1-\beta)\big(\tfrac{q(Q-2)}{2\alpha}-Q\big)>0$, which holds because
$\beta<1$ and $q>\alpha Q^{*}$.
 For the second integral, $\Theta(s)\leq Cs^{Q^{*}}$ and
		\begin{align*}
			\frac{1}{\mu^{Q-2}}\int_{\H \setminus B_{R_\mu}(0)}W_\mu^{Q^{*}}d\xi = O\left(\mu^{(1-\beta)Q-(Q-2)}\right),
		\end{align*}
		which is dominated by $\mu^{2-\frac{q(Q-2)}{4\alpha}}$ precisely when $q>2\alpha\beta Q^{*}$, that is, when $\beta<\frac{q}{2\alpha Q^{*}}$. This is the second constraint imposed on $\beta$ at the start of the proof; it is compatible with $\beta>\frac12$ because $q>4\alpha$ gives $\frac{q}{2\alpha Q^{*}}>\frac{2}{Q^{*}}=\frac{Q-2}{Q}\geq\frac12$ for $Q\geq4$. Hence $I_1+I_2 \to -\infty$, which is Claim 3, and \eqref{eq:E8} completes the proof. \qed
	\end{proof}

	\section{Proof of Theorem \ref{theorem:T}}

	\textbf{Proof of Theorem~\ref{theorem:T}.} From Lemma~\ref{lemma:l4}, $\{v_n\}$ is bounded in $S^2_1(\H)$, so there exists $v\in S^2_1(\H)$ such that $v_n \rightharpoonup v$ in $S^2_1(\H)$, $v_n \to v$ in $L_{loc}^d(\H)$ for $d \in [2, Q^{*})$, and $v_n(\xi) \to v(\xi)$ a.e. in $\H.$ It follows that
	\begin{align*}
		\I^{\prime}(v)w=0 \quad \text{ for any } w \in C_0^{\infty}(\H),
	\end{align*}
	so $v$ is a weak solution of \eqref{eq:P1}. It remains to show that $v$ is nontrivial, and we argue by contradiction: suppose $v\equiv 0.$

	The next result is an instance of the concentration-compactness principle \cite{willem1996minimax}.
	\begin{lemma}{\label{lemma:l8}}
		There exist $r_0, \varrho_0>0$ and a sequence $\{\xi_n\} \subset \H$ such that, up to a subsequence,
		\begin{align*}
			\liminf_{n \to\infty} \int_{B_{r_0}(\xi_n)}\left|v_n\right|^2 d\xi \geq \varrho_0.
		\end{align*}
	\end{lemma}
	\begin{proof}
		Suppose not. Then for every $r>0$ we have $\sup_{\xi \in \H}\int_{B_{r}(\xi)}\left|v_n\right|^2d\xi \to 0$ as $n \to \infty$. Applying Lemma~\ref{lemma:l4} and Lemma~\ref{lemma:l6}, $\{v_n\}$ is bounded in $S^2_1(\H)$ and $v_n \to 0$ in $L^d(\H)$ for $d \in (2,Q^{*})$. Fix $\epsilon>0$. By Remark~\ref{rem:H},
		\begin{align*}
			\int_\H \left|H(\xi,v_n)\right|d\xi \leq \epsilon \int_\H \left(v_n^2 + \left|v_n\right|^{Q^{*}}\right)d\xi + C_\epsilon \int_\H \left|v_n\right|^{\frac{q}{2\alpha}}d\xi \leq \epsilon C + o_n(1),
		\end{align*}
		since $\frac{q}{2\alpha} \in (2, Q^{*})$ and $\{v_n\}$ is bounded in $L^2(\H)\cap L^{Q^{*}}(\H)$. Letting $\epsilon \to 0$ gives
		\begin{align*}
			\lim_{n \to \infty} \int_\H H(\xi,v_n)d\xi=0= \lim_{n \to \infty} \int_\H h(\xi,v_n)v_nd\xi.
		\end{align*}
		Let $A=\lim_{n\to \infty}\int_\H \left(\left|\nah v_n\right|^2 + V(\xi)v_n^2\right)d\xi$ and $B= \lim_{n\to \infty} \int_\H \left|v_n\right|^{Q^{*}} d\xi$, along a further subsequence, both limits existing by boundedness. Then, by the definition of $v_n$,
		\begin{align*}
			c= \frac{A}{2}- \frac{(2\alpha)^{\frac{2}{Q-2}}}{Q^{*}}B,\quad A=(2\alpha)^{\frac{2}{Q-2}}B.
		\end{align*}
		Eliminating $B$ and using $\frac{1}{2}-\frac{1}{Q^{*}}=\frac{1}{Q}$, we get $c=\frac{A}{Q}$. The definition of $S$ gives $S B^{\frac{2}{Q^{*}}} \leq A$, and substituting $B=(2\alpha)^{-\frac{2}{Q-2}}A$ together with $1-\frac{2}{Q^{*}}=\frac{2}{Q}$ and $\frac{2}{Q-2}\cdot\frac{2}{Q^{*}}=\frac{2}{Q}$ yields $A^{\frac{2}{Q}} \geq S(2\alpha)^{-\frac{2}{Q}}$, that is, $A \geq \frac{S^{\frac{Q}{2}}}{2\alpha}.$ Hence
		\begin{align*}
			c =\frac{A}{Q} \geq \frac{S^{\frac{Q}{2}}}{2\alpha Q},
		\end{align*}
		which contradicts Lemma~\ref{lemma:l5}. \qed
	\end{proof}

	We first assume that $(V_2)$ holds. Then, for any $\epsilon>0,$ there exists $\rho>0$ such that $\left|V(\xi)-V_\infty \right|< \epsilon$ for $\varrho(\xi) \geq \rho.$ Consider the functional $I_{\lambda, \infty} : S^2_1(\H) \to \R,$
	\begin{align*}
		I_{\lambda,\infty}(v)= \frac{1}{2} \int_\H \left(\left|\nah v\right|^2 + V_\infty v^2\right)d\xi- \frac{(2\alpha)^{\frac{2}{Q-2}}}{ Q^{*}} \int_\H \left|v\right|^{Q^{*}} d\xi- \int_\H H_\infty(v)d\xi,
	\end{align*}
	where
	\begin{align*}
		H_\infty(v)= \frac{\lambda}{q} \left|g(v)\right|^{q} + \frac{1}{2\alpha Q^{*}} \left|g(v)\right|^{2\alpha Q^{*}}+ \frac{1}{2}V_\infty v^2- \frac{1}{2} V_\infty g^2(v)- \frac{(2\alpha)^{\frac{2}{Q-2}}}{Q^{*}} \left|v\right|^{Q^{*}}.
	\end{align*}
	Using Lemma~\ref{lemma:l1}, we have
	\begin{align*}
		\left|I_{\lambda, \infty}(v_n)-\I(v_n)\right|&= \frac{1}{2} \int_\H \left|V_\infty -V(\xi) \right| \left|g(v_n)\right|^2d\xi \\& \leq \frac{1}{2} \int_{B_\rho(0)} \left|V_\infty -V(\xi) \right| \left|v_n\right|^2d\xi + \frac{1}{2} \int_{\H \setminus {B_\rho(0)}} \left|V_\infty -V(\xi) \right| \left|v_n\right|^2d\xi \\& \leq \|V_\infty - V\|_{L^\infty(B_\rho(0))}\int_{B_\rho(0)}\left|v_n\right|^2d\xi + \frac{\epsilon}{2}\sup_n \left|v_n\right|_2^2.
	\end{align*}
	Since $v_n \to v \equiv 0$ in $L^2(B_\rho(0))$ and $\epsilon>0$ is arbitrary, $\left|I_{\lambda,\infty}(v_n)-\I(v_n)\right| \to 0$ as $n \to \infty.$ Consequently $(1+\|v_n\|)(I^{\prime}_{\lambda, \infty}(v_n)-\I^{\prime}(v_n)) \to 0$, since
	\begin{align*}
		\|I^{\prime}_{\lambda, \infty}(v_n)-\I^{\prime}(v_n)\|=\sup_{\|w\| \leq 1} \left|\int_\H \left(V_\infty-V(\xi)\right)g(v_n)g^{\prime}(v_n)w\,d\xi \right| \to 0 \quad \text{ as } n \to \infty.
	\end{align*}
	Thus $\{v_n\}$ is a bounded Cerami sequence for $I_{\lambda,\infty}$ at level $c$. Let $\{\xi_n\}$ be the sequence given in Lemma~\ref{lemma:l8} and define $z_n(\xi)=v_n(\xi_n \circ \xi).$ Then $\{z_n\}$ is bounded in $S^2_1(\H)$ and
	\begin{align*}
		I_{\lambda, \infty}(z_n) \to c,\quad \left(1+\|z_n\|\right)I^{\prime}_{\lambda, \infty}(z_n) \to 0 \quad \text{ as } n \to \infty.
	\end{align*}
	Therefore, up to a subsequence, there exists $z \in S^2_1(\H)$ with $z_n \rightharpoonup z$ in $S^2_1(\H)$, and $z$ is a critical point of $I_{\lambda,\infty}.$ Since $z_n \to z$ in $L^2(B_{r_0}(0))$ and $\int_{B_{r_0}(0)}\left|z_n\right|^2d\xi \geq \varrho_0$ for large $n$ by Lemma~\ref{lemma:l8}, we get $\int_{B_{r_0}(0)}\left|z\right|^2d\xi \geq \varrho_0$, so $z \neq 0$.

	The constant potential $V_\infty$ satisfies $(V_1)$ and $(V_3)$ with $\kappa=0$, so Lemma \ref{lemma:linf} and Lemma \ref{lemma:decay} apply to $z$: it is bounded, lies in $\Ga^{2,\gamma}_{loc}(\H)$ and decays exponentially in the gauge. In particular Proposition \ref{prop:poho} is available for $z$.

	Let $c_\infty$ be the mountain pass level defined by
	\begin{align*}
		c_\infty = \inf_{\sigma \in \Ga_\infty} \sup_{t\in [0,1]} I_{\lambda, \infty}(\sigma(t)),
	\end{align*}
	where $\Ga_\infty= \{\sigma \in C([0,1], S^2_1(\H)): \sigma(0)=0, \sigma(1)\neq 0, I_{\lambda, \infty}(\sigma(1))<0\}.$

	\textbf{Claim:} $c_\infty \leq I_{\lambda, \infty}(z) \leq c.$ Taking into account \eqref{eq:E4}, Lemma~\ref{lemma:l1} and $2<4\alpha < q,$ we deduce
	\begin{align*}
		g^2(z_n)-g(z_n)g^{\prime}(z_n)z_n\geq 0, \quad \frac{1}{2}f(g(z_n))g^{\prime}(z_n)z_n-F(g(z_n)) \geq 0 \quad \text{ for all } n\in \N.
	\end{align*}
	Thus, by Fatou's lemma,
	\begin{align*}
		c&=\limsup_{n \to \infty}\left(I_{\lambda, \infty}(z_n)- \frac{1}{2}I^{\prime}_{\lambda, \infty}(z_n)z_n \right) \\&=\limsup_{n \to \infty} \left( \int_\H \frac{V_\infty}{2}\left(g^2(z_n)-g(z_n)g^{\prime}(z_n)z_n\right) d\xi+ \int_\H \left(\frac{1}{2}f(g(z_n))g^{\prime}(z_n)z_n-F(g(z_n))\right)d\xi\right) \\& \geq \int_\H \frac{V_\infty}{2}\left(g^2(z)-g(z)g^{\prime}(z)z\right)d\xi+ \int_\H \left(\frac{1}{2}f(g(z))g^{\prime}(z)z-F(g(z))\right)d\xi \\&= I_{\lambda, \infty}(z)- \frac{1}{2}I^{\prime}_{\lambda, \infty}(z)z= I_{\lambda, \infty}(z).
	\end{align*}

	Next we construct a path $\sigma$ with
	\begin{align*}
		&\sigma(0)=0,\quad z \in \sigma([0,1]),\quad I_{\lambda, \infty}(\sigma(1))<0, \\& \sigma(t) \neq 0 \quad \text{ for } t\in (0,1],\\& I_{\lambda, \infty}(z)=\max_{t\in [0,1]} I_{\lambda, \infty}(\sigma(t)).
	\end{align*}
	Define
	\begin{align*}
		z_t(\xi)=
		\begin{cases}
			z\left(\delta_{1/t}\,\xi\right) & \text{ if } t>0, \\
			0 & \text{ if } t=0.
		\end{cases}
	\end{align*}
	Pick $t_1 \in (0,1)$, $t_2 \in (1,\infty)$ and $t_3 >t_2$ such that $\sigma$ can be defined in three parts,
	\begin{align*}
		&\sigma_1:[0,1] \to S^2_1(\H),\quad \sigma_1(\ba)=\ba z_{t_1}; \\& \sigma_2:[t_1,t_2] \to S^2_1(\H),\quad \sigma_2(t)=z_{t}; \\& \sigma_3:[1,t_3] \to S^2_1(\H),\quad \sigma_3(\ba)=\ba z_{t_2}.
	\end{align*}
	Since $z$ is a weak solution of $-\Delta_{\mathbb{H}}z=k(z)$ with $k(s)=[f(g(s))-V_\infty g(s)]g^\prime(s)$, we have $\int_\H k(z)z\, d\xi= \int_\H \left|\nah z\right|^2 d\xi>0.$ As a consequence, there exists $t_3>0$ such that
	\begin{equation}{\label{eq:E12}}
		\int_\H k(\ba z)z\, d\xi>0 \quad \text{ for all } \ba \in [1,t_3].
	\end{equation}
	Define $\varphi(t)= \frac{k(t)}{t}$ for $t>0.$ Then, by \eqref{eq:E12},
	\begin{equation}{\label{eq:E13}}
		\int_\H \varphi(\ba z)z^2 d\xi>0 \quad \text{ for all } \ba \in [1,t_3].
	\end{equation}
	Using the scaling relations for $\delta_t$ yields
	\begin{align*}
		\frac{d}{d\ba} I_{\lambda, \infty}(\ba z_t)=\ba t^{Q-2} \left(\int_\H \left|\nah z\right|^2 d\xi- t^2 \int_\H \varphi(\ba z)z^2 d\xi\right).
	\end{align*}
	As a result, there exists $t_1 \in (0,1)$ such that
	\begin{equation}{\label{eq:E14}}
		\|\nah z\|_2^2-t_1^2 \int_\H \varphi(\ba z)z^2 d\xi>0 \quad \text{ for all } \ba \in [0,1].
	\end{equation}
	Using \eqref{eq:E13}, we choose $t_2 >1$ in such a way that
	\begin{equation}{\label{eq:E15}}
		\|\nah z\|_2^2-t_2^2 \int_\H \varphi(\ba z)z^2 d\xi < \frac{2}{1-t_3^2} \|\nah z\|_2^2 \quad \text{ for all } \ba \in [1, t_3].
	\end{equation}
	With the help of \eqref{eq:E14}, the map $\ba \mapsto I_{\lambda, \infty}(\ba z_{t_1})$ increases and attains its maximum at $\ba =1$ along $\sigma_1.$ Let $K(t)=\int_0^t k(s)ds.$ Applying Proposition~\ref{prop:poho} with $V \equiv V_\infty$, we get $\int_\H K(z)d\xi = \frac{Q-2}{2Q}\|\nah z\|_2^2$, whence
	\begin{align*}
		I_{\lambda,\infty}(z_t)= \frac{t^{Q-2}}{2}\|\nah z\|_2^2 - t^{Q}\,\frac{Q-2}{2Q}\|\nah z\|_2^2 \leq \frac{1}{Q}\|\nah z\|_2^2 = I_{\lambda, \infty}(z),
	\end{align*}
	with equality if and only if $t=1$, along $\sigma_2.$

	By \eqref{eq:E15}, $\ba \mapsto I_{\lambda, \infty}(\ba z_{t_2})$ decreases along $\sigma_3$, so
	\begin{align*}
		I_{\lambda, \infty}(\sigma_3(\ba)) \leq I_{\lambda, \infty}(z_{t_2}) \leq I_{\lambda, \infty}(z), \qquad I_{\lambda, \infty}(\sigma_1(\ba)) \leq I_{\lambda, \infty}(z_{t_1}) \leq I_{\lambda, \infty}(z),
	\end{align*}
	and hence $I_{\lambda, \infty}(z)=\max_{t\in [0,t_3]} I_{\lambda, \infty}(\sigma(t)).$ From \eqref{eq:E15}, using $t_2^{Q-2}>1$ and the negativity of the bracket, along $\sigma_3$ we deduce
	\begin{align*}
		I_{\lambda, \infty}(t_3 z_{t_2})&= I_{\lambda, \infty}( z_{t_2}) + \int_1^{t_3} \frac{d}{d\ba} I_{\lambda, \infty}(\ba z_{t_2})d\ba\\& \leq \frac{1}{Q} \|\nah z\|^2+\int_1^{t_3} \frac{2\ba}{1-t_3^2} \|\nah z\|^2d\ba= \left(\frac{1}{Q}-1\right) \|\nah z\|^2 <0.
	\end{align*}
	The path $\sigma$, reparametrised over $[0,1]$, together with $c_\infty$ gives $c_\infty \leq \max_{t \in [0,1]} I_{\lambda, \infty}(\sigma(t))=I_{\lambda, \infty}(z).$

	Since $V \leq V_\infty$ we have $\I \leq I_{\lambda, \infty}$ on $S^2_1(\H)$, so $I_{\lambda,\infty}(\sigma(1))<0$ forces $\I(\sigma(1))<0$ and hence $\sigma \in \Ga_\infty \subset \Ga.$ The function $t \mapsto \I(\sigma(t))$ is continuous on $[0,1]$, so its supremum is attained at some $t^{\prime}$; since $\I(\sigma(0))=0<c$ by \eqref{eq:cpos}, we have $t^{\prime} \in (0,1]$ and therefore $\sigma(t^{\prime}) \neq 0$ by construction of $\sigma$. As $V \not\equiv V_\infty$,
	\begin{align*}
		I_{\lambda, \infty}(\sigma(t^{\prime}))-\I(\sigma(t^{\prime}))=\frac{1}{2}\int_\H \left(V_\infty-V(\xi)\right)g^2(\sigma(t^{\prime}))d\xi>0.
	\end{align*}
	Combining with $c_\infty \leq I_{\lambda, \infty}(z) \leq c$, we obtain
	\begin{align*}
		c \leq \sup_{t \in [0,1]} \I(\sigma(t))=\I(\sigma(t^{\prime}))< I_{\lambda, \infty}(\sigma(t^{\prime})) \leq \max_{t\in [0,1]} I_{\lambda, \infty}(\sigma(t))=I_{\lambda, \infty}(z) \leq c,
	\end{align*}
	which is absurd. Hence $v$ is nontrivial.

	Now suppose that $(V^{\prime}_2)$ holds, that is, $V(\gamma \circ \xi)=V(\xi)$ for every $\gamma \in \Ga_{\mathbb{Z}}$ and every $\xi \in \H$. Let $\F$ be a compact fundamental domain for $\Ga_{\mathbb{Z}}$, and for each $n$ choose $\gamma_n \in \Ga_{\mathbb{Z}}$ with $\gamma_n^{-1}\circ \xi_n \in \F$, where $\{\xi_n\}$ is the sequence of Lemma~\ref{lemma:l8}. Set $\tilde{z}_n(\xi)=v_n(\gamma_n \circ \xi)$. Since $\I$ is invariant under composition with $\tau_{\gamma_n}$, $\{\tilde{z}_n\}$ is again a bounded Cerami sequence for $\I$ at level $c$. Hence there is a subsequence, denoted again by $\{\tilde{z}_n\}$, and $\tilde{z} \in S^2_1(\H)$ with $\tilde{z}_n \rightharpoonup \tilde{z}$ in $S^2_1(\H)$, and $\tilde{z}$ is a weak solution of \eqref{eq:P1}. Moreover $\gamma_n^{-1}\circ\xi_n$ ranges in the compact set $\F$, so all the balls $B_{r_0}(\gamma_n^{-1}\circ \xi_n)$ are contained in a fixed gauge ball $B_{R_0}(0)$, and Lemma~\ref{lemma:l8} together with $\tilde{z}_n \to \tilde{z}$ in $L^2(B_{R_0}(0))$ gives $\int_{B_{R_0}(0)}\left|\tilde{z}\right|^2d\xi \geq \varrho_0$. Hence $\tilde{z}$ is nontrivial. \QED

	\section{Nonexistence}

	In this section we show that the equation has no nontrivial solution once the nonlinearity fails to be subcritical relative to $2\alpha Q^{*}$, under a suitable condition on $V$ with respect to the anisotropic dilations.

	We combine the two identities of Section 3. Write
	\begin{align*}
		E_1 = \int_\H \left|\nah u\right|^2 d\xi, \quad E_2 = \int_\H \left|u\right|^{4\alpha-2}\left|\nah u\right|^2 d\xi, \quad \mathcal{V} = \int_\H V u^2 d\xi, \quad D=\int_\H (ZV)u^2 d\xi.
	\end{align*}

	\begin{prop}{\label{prop:pencil}}
		Assume $(V_1)$ and $(V_3)$, together with $(f_1)$ and $(f_2)$, and let $u=g(v)$ with $v \in S^2_1(\H)$ a weak solution of \eqref{eq:P1}. Then, for every $\ga>0$,
		\begin{equation}{\label{eq:pencil}}
			\int_\H \left[f(u)u - \ga F(u)\right]d\xi = \left(1-\frac{\ga(Q-2)}{2Q}\right)E_1 + \left(4\alpha^2 - \frac{\ga\alpha(Q-2)}{Q}\right)E_2 + \left(1-\frac{\ga}{2}\right)\mathcal{V} - \frac{\ga}{2Q}D.
		\end{equation}
		The coefficient of $E_2$ vanishes if and only if $\ga = 2\alpha Q^{*}$, and for that value
		\begin{equation}{\label{eq:star}}
			\int_\H \left[f(u)u - 2\alpha Q^{*} F(u)\right]d\xi = (1-2\alpha)E_1 + \left(1-\alpha Q^{*}\right)\mathcal{V} - \frac{2\alpha}{Q-2}D.
		\end{equation}
	\end{prop}
	\begin{proof}
		Proposition \ref{prop:poho} gives
		\begin{align*}
			\int_\H F(u)d\xi = \frac{1}{Q}\left[\frac{Q-2}{2}E_1 + \alpha(Q-2)E_2 + \frac{Q}{2}\mathcal{V} + \frac{D}{2}\right],
		\end{align*}
		and subtracting $\ga$ times this from Lemma \ref{lemma:nehari} yields \eqref{eq:pencil}. The coefficient of $E_2$ vanishes precisely when
		\begin{align*}
			\ga = \frac{4\alpha Q}{Q-2}= 2\alpha \cdot \frac{2Q}{Q-2} = 2\alpha Q^{*},
		\end{align*}
		and substituting this value gives $1-\frac{\ga(Q-2)}{2Q}=1-2\alpha$, $1-\frac{\ga}{2}=1-\alpha Q^{*}$ and $-\frac{\ga}{2Q}=-\frac{2\alpha}{Q-2}$, which is \eqref{eq:star}. \qed
	\end{proof}

	\begin{rem}
		Proposition \ref{prop:pencil} explains the exponent $2\alpha Q^{*}$: it is the only value of $\ga$ for which the quasilinear energy $E_2$ drops out of the Poho\v{z}aev--Nehari pencil. No control on $E_2$ is therefore needed below, and this is what makes the nonexistence result available at all.
	\end{rem}

	\textbf{Proof of Theorem \ref{thm:nonex} :}
		By $(V_3)$, we have \[D \geq -\kappa \mathcal{V} \text{ and }  \frac{2\alpha\kappa}{Q-2} \leq \alpha Q^{*}-1.\] Then, \eqref{eq:star} yields
		\begin{align*}
			0 \leq \int_\H \left[f(u)u-2\alpha Q^{*}F(u)\right]d\xi
			\leq (1-2\alpha)E_1 + \left(1-\alpha Q^{*}+\frac{2\alpha \kappa}{Q-2}\right)\mathcal{V}
			\leq (1-2\alpha)E_1 \leq 0,
		\end{align*}
		as $\alpha>\frac12$. Hence $E_1=0$ or equivalently, $\nah u = 0$ a.e.. Since
		$\H$ is connected and the fields $X_i, Y_i$ satisfy H\"ormander's condition,
		$u$ is constant, and $u \in L^2(\H)$ forces $u \equiv 0$. \qed

	\begin{co}{\label{cor:lambda}}
		Assume $(V_1)$ and $(V_3)$, and let $2<q<2\alpha Q^{*}$ with $\lambda \leq 0$. Then
		\begin{align*}
			-\Delta_{\mathbb{H}} u + V(\xi)u - \Delta_{\mathbb{H}}\!\left(\left|u\right|^{2\alpha}\right)\left|u\right|^{2\alpha-2}u = \lambda \left|u\right|^{q-2}u + \left|u\right|^{2\alpha Q^{*}-2}u \quad \text{ in } \H
		\end{align*}
		has no nontrivial weak solution $u=g(v)$ with $v \in S^2_1(\H)$.
	\end{co}
	\begin{proof}
		With $f(t)=\lambda\left|t\right|^{q-2}t+\left|t\right|^{2\alpha Q^{*}-2}t$, one has
		\begin{align*}
			f(s)s - 2\alpha Q^{*}F(s)= \lambda\left(1-\frac{2\alpha Q^{*}}{q}\right)\left|s\right|^{q}\geq 0,
		\end{align*}
		as $\lambda \leq 0$ and $q<2\alpha Q^{*}$; the critical term cancels identically. Both $(f_1)$ and $(f_2)$ hold, so Theorem \ref{thm:nonex} applies. \qed
	\end{proof}

	\begin{rem}{\label{rem:super}}
		Corollary \ref{cor:lambda} is stated for $p=2\alpha Q^{*}$ because $(f_1)$ fails for $p>2\alpha Q^{*}$, and with it the regularity theory of Section 3. In the supercritical range, the conclusion of Theorem \ref{thm:nonex} remains valid for solutions satisfying
		\begin{align*}
			\left(1+\left|u\right|^{4\alpha-2}\right)\left|\nah u\right|^2 \in L^1(\H), \quad \left(V+\left|ZV\right|\right)u^2 \in L^1(\H), \quad \left|F(u)\right|+\left|f(u)u\right| \in L^1(\H),
		\end{align*}
		together with $\liminf_{R \to \infty}R\int_{A_R}\left|Tv\right|\left|\nah v\right|d\xi=0$: these are exactly the properties of $u$ used in the proofs of Proposition \ref{prop:poho} and Lemma \ref{lemma:nehari}. The corresponding assumption in the Euclidean case is $u^2\left|\nabla u\right|^2 \in L^1(\R^m)$ \cite{liu2004solutions}.
	\end{rem}

	\begin{rem}{\label{rem:V4}}
		Written out, $(V_3)$ reads
		\begin{align*}
			\sum_{i=1}^{N}\left(x_i V_{x_i}+y_iV_{y_i}\right)+2tV_t \geq -\kappa V,
		\end{align*}
		the weight $2$ on the centre variable reflects the anisotropy of $\delta_\theta$.
		The value $\kappa=0$ gives $ZV \geq 0$, that is, $V$ nondecreasing along each dilation orbit, which is the sub-Riemannian analogue of the assumption $\nabla V(x)\cdot x \geq 0$ imposed in \cite{liu2004solutions}.

		\emph{(a) $(V_1)$ and $(V_2)$ do not imply $(V_3)$.} Let
		\begin{align*}
			\varsigma(\xi)=\varrho(\xi)^{4}=t^2+\left(\left|x\right|^2+\left|y\right|^2\right)^2 ,
		\end{align*}
		which is $\delta_\theta$-homogeneous of degree $4$, so that $Z\varsigma=4\varsigma$ and $Z(W\circ \varsigma)=4\varsigma\,W^{\prime}(\varsigma)$ for every $W \in C^1([0,\infty))$. Take
		\begin{align*}
			W(r)=V_\infty - a\,e^{-b(r-1)^2}, \qquad 0<a<V_\infty, \ b>0,
		\end{align*}
		and put $V=W\circ\varsigma$. Then $V_\infty-a \leq V \leq V_\infty$ and $V(\xi)\to V_\infty$ as $\varrho(\xi)\to\infty$, so $(V_1)$ and $(V_2)$ hold. On the other hand $W^{\prime}(r)=2ab(r-1)e^{-b(r-1)^2}$, so at $r_b=1-b^{-1/2}$ we have
		\begin{align*}
			4r_b W^{\prime}(r_b)+\kappa W(r_b) \leq -8a\,e^{-1}\sqrt{b}\left(1-b^{-1/2}\right)+\kappa V_\infty ,
		\end{align*}
		which is negative for $b$ large. Hence $(V_3)$ fails for every admissible $\kappa$.

		\emph{(b) The conditions are nevertheless simultaneously satisfiable.} For
		\begin{align*}
			W(r)=V_\infty - \frac{a}{1+r}, \qquad 0<a<V_\infty,
		\end{align*}
		one has $ZV=4a\varsigma(1+\varsigma)^{-2}\geq 0$, so $(V_3)$ holds with $\kappa=0$, together with $V\geq V_\infty-a>0$, $V \leq V_\infty$, $V\not\equiv V_\infty$ and $V(\xi)\to V_\infty$ as $\varrho(\xi)\to\infty$. Note that $(V_3)$ constrains $V$ only along the dilation orbits and places no restriction on its behaviour over a gauge sphere; potentials of the form $W\circ \varsigma$ are examples, not the general case.
		\emph{(c) The periodic case is excluded.} Let $V \in C^1(\H)$ be
		$\Ga_{\mathbb{Z}}$-invariant and nonconstant. If $V_t\not\equiv 0$, pick
		$\xi_0=(\zeta_0,t_0)$ with $V_t(\xi_0)<0$; invariance under $(0,0,k)$ gives
		$ZV(\zeta_0,t_0+k)=\zeta_0\cdot\nabla_\zeta V(\xi_0)+2(t_0+k)V_t(\xi_0)\to-\infty$
		as $k \to \infty$. If $V_t\equiv0$, then $V=V(\zeta)$ with $\nabla_\zeta V$ periodic;
		choosing $\zeta_0$ with $\nabla_\zeta V(\zeta_0)\neq0$ and a lattice vector $w$ with
		$w\cdot \nabla_\zeta V(\zeta_0)\neq 0$, we get
		$ZV(\zeta_0+kw)=\zeta_0\cdot\nabla_\zeta V(\zeta_0)+k\,w\cdot\nabla_\zeta V(\zeta_0)$,
		which tends to $-\infty$ along a suitable sequence $k \to \pm\infty$. In either case $ZV$
		is unbounded below while $V$ is bounded, so no admissible $\kappa$ makes $(V_3)$ hold.
		Corollary~\ref{cor:lambda} therefore does not overlap the case $(V_2^\prime)$, and this is
		a feature of the dilation structure rather than a defect of the proof: the generator $Z$
		has coefficients growing linearly in $\varrho$, against which a
		periodic gradient cannot hold a sign.
	\end{rem}

	\noindent \textbf{Funding} \\
	The authors declare that no funds, grants, or other support were received during the preparation of this manuscript.

	\noindent \textbf{Competing Interests} \\
	The authors declare that there is no conflict of interest.
	\bibliographystyle{siam}
	\bibliography{ref}

@article{wu2014multiple,
  title={Multiple solutions for quasilinear {S}chr{\"o}dinger equations with a parameter},
  author={Wu, Xian},
  journal={Journal of Differential Equations},
  volume={256},
  number={7},
  pages={2619--2632},
  year={2014},
  publisher={Elsevier}
}

@book{schechter2012linking,
  title={Linking methods in critical point theory},
  author={Schechter, Martin},
  year={2012},
  publisher={Springer Science \& Business Media}
}

@article{brezis1983positive,
  title={Positive solutions of nonlinear elliptic equations involving critical {S}obolev exponents},
  author={Br{\'e}zis, Ha{\"\i}m and Nirenberg, Louis},
  journal={Communications on pure and applied mathematics},
  volume={36},
  number={4},
  pages={437--477},
  year={1983},
  publisher={Wiley Subscription Services, Inc., A Wiley Company, New York}
}

@article{citti1995semilinear,
  title={Semilinear {D}irichlet problem involving critical exponent for the {K}ohn {L}aplacian},
  author={Citti, G},
  journal={Annali di Matematica Pura ed Applicata},
  volume={169},
  number={1},
  pages={375--392},
  year={1995},
  publisher={Springer}
}

@book{heisenberg2013physical,
  title={The physical principles of the quantum theory},
  author={Heisenberg, Werner},
  year={2013},
  publisher={Courier Corporation}
}

@article{hormander1967hypoelliptic,
  title   = {Hypoelliptic second order differential equations},
  author  = {H{\"o}rmander, Lars},
  journal = {Acta Mathematica},
  volume  = {119},
  pages   = {147--171},
  year    = {1967},
  publisher = {Springer}
}

@article{folland1975subelliptic,
  title={Subelliptic estimates and function spaces on nilpotent {L}ie groups},
  author={Folland, Gerald B},
  journal={Arkiv f{\"o}r Matematik},
  volume={13},
  number={1},
  pages={161--207},
  year={1975},
  publisher={Kluwer Academic Publishers Dordrecht}
}

@article{liu2003soliton,
  title={Soliton solutions for quasilinear {S}chr{\"o}dinger equations, I},
  author={Liu, Jiaquan and Wang, Zhi-Qiang},
  journal={Proceedings of the American Mathematical Society},
  pages={441--448},
  year={2003},
  publisher={JSTOR}
}

@article{adachi2012uniqueness,
  title={Uniqueness of the ground state solutions of quasilinear {S}chr{\"o}dinger equations},
  author={Adachi, Shinji and Watanabe, Tatsuya},
  journal={Nonlinear Analysis: Theory, Methods \& Applications},
  volume={75},
  number={2},
  pages={819--833},
  year={2012},
  publisher={Elsevier}
}

@article{poppenberg2002existence,
  title={On the existence of soliton solutions to quasilinear {S}chr{\"o}dinger equations},
  author={Poppenberg, Markus and Schmitt, Klaus and Wang, Zhi-Qiang},
  journal={Calculus of Variations and Partial Differential Equations},
  volume={14},
  number={3},
  pages={329--344},
  year={2002},
  publisher={Springer}
}

@article{colin2004solutions,
  title={Solutions for a quasilinear {S}chr{\"o}dinger equation: a dual approach},
  author={Colin, Mathieu and Jeanjean, Louis},
  journal={Nonlinear Analysis: Theory, Methods \& Applications},
  volume={56},
  number={2},
  pages={213--226},
  year={2004},
  publisher={Elsevier}
}

@misc {etde_5937557,
title = {Nonlinear scalar field equations. Pt. 1},
author = {Berestycki, H. and Lions, P.L.}, 
abstractNote = {This paper, as well as a subsequent one, is concerned with the existence of nontrivial solutions for some semi-linear elliptic equations in Rsup(N). Such problems are motivated in particular by the search for certain kinds of solitary waves (stationary states) in nonlinear equations of the Klein-Gordon or {S}chrodinger type.},
journal = {Archive for Rational Mechanics and Analysis},
volume = {82:4},
place = {Germany},
year = {1983},
month = {Jan}
}

@article{berestycki1983equations,
  title={{\'E}quations de champs scalaires {E}uclidiens non lin{\'e}aires dans le plan (French)[Nonlinear {E}uclidean scalar field equations in the plane]},
  author={Berestycki, Henri},
  journal={Comptes Rendus de l'Acad{\'e}mie des Sciences, S{\'e}rie I - Math{\'e}matique},
  volume={297},
  pages={307},
  year={1983}
}

@article{liu2004solutions,
  title={Solutions for quasilinear {S}chr{\"o}dinger equations via the {N}ehari method},
  author={Liu, Jia-quan and Wang, Ya-qi and Wang, Zhi-Qiang},
  journal={Communications in Partial Differential Equations},
  volume={29},
  number={5-6},
  pages={879--901},
  year={2004},
  publisher={Taylor \& Francis}
}

@article{pucci1986general,
  title={A general variational identity},
  author={Pucci, Patrizia and Serrin, James},
  journal={Indiana University Mathematics Journal},
  volume={35},
  number={3},
  pages={681--703},
  year={1986},
  publisher={JSTOR}
}

@article{miyagaki2010soliton,
  title={Soliton solutions for quasilinear {S}chr{\"o}dinger equations with critical growth},
  author={Miyagaki, Ol{\'\i}mpio H and Soares, S{\'e}rgio HM and others},
  journal={Journal of Differential Equations},
  volume={248},
  number={4},
  pages={722--744},
  year={2010},
  publisher={Elsevier}
}

@inproceedings{lions1984concentration,
  title={The concentration-compactness principle in the calculus of variations. The locally compact case, part 2 },
  author={Lions, Pierre-Louis},
  booktitle={Annales de l'Institut Henri Poincar{\'e} C, Analyse non lin{\'e}aire},
  volume={1},
  pages={223--283},
  year={1984},
  organization={Elsevier}
}

@article{alves2007soliton,
  title={Soliton solutions to a class of quasilinear elliptic equations on $\mathbb{R}$ },
  author={Alves, MJ and Carriao, PC and Miyagaki, OH},
  journal={Advanced Nonlinear Studies},
  volume={7},
  number={4},
  pages={579--597},
  year={2007},
  publisher={Advanced Nonlinear Studies, Inc.}
}

@article{chen2015multiple,
  title={Multiple solutions for a class of quasilinear {S}chr{\"o}dinger equations in $\mathbb{R}^{N}$ },
  author={Chen, Caisheng},
  journal={Journal of Mathematical Physics},
  volume={56},
  number={7},
  year={2015},
  publisher={AIP Publishing}
}

@article{kurihara1981large,
  title={Large-amplitude quasi-solitons in superfluid films},
  author={Kurihara, Susumu},
  journal={Journal of the Physical Society of Japan},
  volume={50},
  number={10},
  pages={3262--3267},
  year={1981},
  publisher={The Physical Society of Japan}
}

@article{medeiros2008nonhomogeneous,
  title={A nonhomogeneous elliptic problem involving critical growth in dimension two},
  author={Medeiros, Everaldo and Severo, Uberlandio and others},
  journal={Journal of Mathematical Analysis and Applications},
  volume={345},
  number={1},
  pages={286--304},
  year={2008},
  publisher={Elsevier}
}

@article{giacomoni2005multiplicity,
  title={A multiplicity result to a nonhomogeneous elliptic equation in whole space $\mathbb{R}^2$},
  author={Giacomoni, J and Sreenadh, K},
  journal={Advances in Mathematical Sciences and Applications},
  volume={15},
  number={2},
  pages={467},
  year={2005},
  publisher={MARUZEN CO. LTD}
}

@article{lam2012existence,
  title={Existence and multiplicity of solutions to equations of ${N}$-{L}aplacian type with critical exponential growth in $\mathbb{R}^{N}$},
  author={Lam, Nguyen and Lu, Guozhen},
  journal={Journal of Functional Analysis},
  volume={262},
  number={3},
  pages={1132--1165},
  year={2012},
  publisher={Elsevier}
}

@article{floer1986nonspreading,
  title={Nonspreading wave packets for the cubic {S}chr{\"o}dinger equation with a bounded potential},
  author={Floer, Andreas and Weinstein, Alan},
  journal={Journal of Functional Analysis},
  volume={69},
  number={3},
  pages={397--408},
  year={1986},
  publisher={Elsevier}
}

@article{rabinowitz1992class,
  title={On a class of nonlinear {S}chr{\"o}dinger equations},
  author={Rabinowitz, Paul H},
  journal={Zeitschrift Angewandte Mathematik und Physik},
  volume={43},
  number={2},
  pages={270--291},
  year={1992}
}

@article{strauss1977existence,
  title={Existence of solitary waves in higher dimensions},
  author={Strauss, Walter A},
  journal={Communications in Mathematical Physics},
  volume={55},
  number={2},
  pages={149--162},
  year={1977},
  publisher={Springer}
}

@article{miyagaki2007soliton,
  title={Soliton solutions for quasilinear {S}chr{\"o}dinger equations: the critical exponential case},
  author={Miyagaki, Ol{\'\i}mpio H and Soares, S{\'e}rgio HM and others},
  journal={Nonlinear Analysis: Theory, Methods \& Applications},
  volume={67},
  number={12},
  pages={3357--3372},
  year={2007},
  publisher={Elsevier}
}

@inproceedings{garofalo1990frequency,
  title={Frequency functions on the {H}eisenberg group, the uncertainty principle and unique continuation},
  author={Garofalo, Nicola and Lanconelli, Ermanno},
  booktitle={Annales de l'institut Fourier},
  volume={40},
  pages={313--356},
  year={1990}
}

@book{ivanov2011extremals,
  title={Extremals for the {S}obolev inequality and the quaternionic contact {Y}amabe problem},
  author={Ivanov, Stefan P and Vassilev, Dimiter N},
  year={2011},
  publisher={World Scientific}
}

@article{willem1996minimax,
  title={Minimax theorems},
  author={Willem, M},
  journal={Progress in Nonlinear Differential Equations and their Applications},
  volume={24},
  year={1996}
}

@article{moameni2007class,
  title={On a class of periodic quasilinear {S}chr{\"o}dinger equations involving critical growth in $\mathbb{R}^2$},
  author={Moameni, Abbas},
  journal={Journal of Mathematical Analysis and Applications},
  volume={334},
  number={2},
  pages={775--786},
  year={2007},
  publisher={Elsevier}
}

@article{do2009semi,
  title={Semi-classical states for quasilinear {S}chr{\"o}dinger equations arising in plasma physics},
  author={do O, Joao Marcos and Moameni, Abbas and Severo, Uberlandio},
  journal={Communications in Contemporary Mathematics},
  volume={11},
  number={04},
  pages={547--583},
  year={2009},
  publisher={World Scientific}
}

@article{folland1974estimates,
  title={Estimates for the article empty $\bar{\pa_b}$ complex and analysis on the {H}eisenberg group},
  author={Folland, Gerald B and Stein, Elias M},
  journal={Communications on Pure and Applied Mathematics},
  volume={27},
  number={4},
  pages={429--522},
  year={1974},
  publisher={Wiley Online Library}
}

@article {folland1973fundamental,
    AUTHOR = {Folland, G. B.},
     TITLE = {A fundamental solution for a subelliptic operator},
  JOURNAL = {Bulletin of the American Mathematical Society},
    VOLUME = {79},
      YEAR = {1973},
     PAGES = {373--376},
}

@article{li2015positive,
  title={Positive solution for quasilinear {S}chr{\"o}dinger equations with a parameter},
  author={Li, GB},
  journal={Commun. Pure Appl. Anal},
  volume={14},
  number={5},
  pages={1803--1816},
  year={2015}
}

@article{garofalo1992existence,
  author  = {Garofalo, Nicola and Lanconelli, Ermanno},
  title   = {Existence and nonexistence results for semilinear equations on the {H}eisenberg group},
  journal = {Indiana University Mathematics Journal},
  volume  = {41},
  number  = {1},
  pages   = {71--98},
  year    = {1992}
}

@article{lanconelli2000nonexistence,
  author  = {Lanconelli, Ermanno and Uguzzoni, Francesco},
  title   = {Non-existence results for semilinear {K}ohn--{L}aplace equations in unbounded domains},
  journal = {Communications in Partial Differential Equations},
  volume  = {25},
  number  = {9-10},
  pages   = {1703--1739},
  year    = {2000}
}

@article{loiudice2007semilinear,
  author  = {Loiudice, Annunziata},
  title   = {Semilinear subelliptic problems with critical growth on {C}arnot groups},
  journal = {Manuscripta Mathematica},
  volume  = {124},
  number  = {2},
  pages   = {247--259},
  year    = {2007}
}

@article{shen2013soliton,
  title={Soliton solutions for generalized quasilinear {S}chr{\"o}dinger equations},
  author={Shen, Yaotian and Wang, Youjun},
  journal={Nonlinear Analysis: Theory, Methods \& Applications},
  volume={80},
  pages={194--201},
  year={2013},
  publisher={Elsevier}
}
\end{document}